\documentclass[11pt]{amsart}

\usepackage{amsmath,amssymb,amsthm,mathtools}
\usepackage[mathscr]{euscript}
\usepackage{color}
\usepackage{epsf}
\usepackage{enumerate}
\usepackage{graphicx}
\usepackage{array}
\usepackage{bbm}
\usepackage{pgf,tikz}
\usepackage{mathrsfs}
\usetikzlibrary{arrows}

\usepackage[margin=1in]{geometry}

\theoremstyle{definition}
\newtheorem{theorem}{Theorem}[section]
\newtheorem{prop}[theorem]{Proposition}

\newtheorem{conjecture}[theorem]{Conjecture}
\newtheorem{lemma}[theorem]{Lemma}
\newtheorem{definition}[theorem]{Definition}
\newtheorem{corollary}[theorem]{Corollary}
\newtheorem{remark}[theorem]{Remark}
\newtheorem{example}[theorem]{Example}

\newcommand{\conv}[1]{\mathrm{conv}\left(#1\right)}

\newcommand{\R}{\mathbb{R}}
\newcommand{\Z}{\mathbb{Z}}

\newcommand{\Q}{\mathbb{Q}}
\newcommand{\rk}{\text{rk }}

\newcommand{\cone}[1]{\mathrm{cone}\left(#1\right)}

\newcommand{\1}{\mathbbm{1}}
\def\ss{\textsuperscript}

\newcommand{\note}[1]{\par \noindent
  \framebox{\begin{minipage}[c]{0.95 \textwidth} NOTE:
      #1 \end{minipage}}\par}
\newcommand\commentout[1]{}

\begin{document}



\title{Laplacian Simplices}

\author{Benjamin Braun}
\address{Department of Mathematics\\
         University of Kentucky\\
         Lexington, KY 40506--0027}
\email{benjamin.braun@uky.edu}

\author{Marie Meyer}
\address{Department of Mathematics\\
         University of Kentucky\\
         Lexington, KY 40506--0027}
\email{marie.meyer@uky.edu}



\date{21 June 2017}

\thanks{
The first author was partially supported by grant H98230-16-1-0045 from the U.S. National Security Agency.
The authors thank Tefjol Pllaha, Liam Solus, Akiyoshi Tsuchiya, and Devin Willmott for their comments and suggestions.
}

\begin{abstract}
This paper initiates the study of the \emph{Laplacian simplex} $T_G$ obtained from a finite graph $G$ by taking the convex hull of the columns of the Laplacian matrix for $G$.
Basic properties of these simplices are established, and then a systematic investigation of $T_G$ for trees, cycles, and complete graphs is provided.
Motivated by a conjecture of Hibi and Ohsugi, our investigation focuses on reflexivity, the integer decomposition property, and unimodality of Ehrhart $h^*$-vectors.
We prove that if $G$ is a tree, odd cycle, complete graph, or a whiskering of an even cycle, then $T_G$ is reflexive.
We show that while $T_{K_n}$ has the integer decomposition property, $T_{C_n}$ for odd cycles does not.
The Ehrhart $h^*$-vectors of $T_G$ for trees, odd cycles, and complete graphs are shown to be unimodal.
As a special case it is shown that when $n$ is an odd prime, the Ehrhart $h^*$-vector of $T_{C_n}$ is given by $(h_0^*,\ldots,h_{n-1}^*)=(1,\ldots,1,n^2-n+1,1,\ldots, 1)$.
We also provide a combinatorial interpretation of the Ehrhart $h^*$-vector for $T_{K_n}$.
\end{abstract}

\maketitle

\tableofcontents

\newpage

\section{Introduction}\label{sec:intro}

\subsection{Motivation}
Let $G$ be a finite graph on the vertex set $\{1,2,\ldots,n\}$.
There are many profitable ways to associate a polytope to $G$.
One well-known example is the \emph{edge polytope} of $G$, obtained by taking the convex hull of the vectors $e_i+e_j$ for each edge $\{i,j\}$ in $G$, where $e_i$ denotes the $i\ss{th}$ standard basis vector in $\R^n$.
Equivalently, the edge polytope is the convex hull of the columns of the unsigned vertex-edge incidence matrix of $G$.
Many geometric, combinatorial, and algebraic properties of edge polytopes have been established over the past several decades, e.g.~\cite{HibiEhrhartEdge,hibiohsugiedgepolytope,TranZiegler,VillarrealEdgePolytopes}.
Another well-known matrix associated with a graph $G$ is the Laplacian $L(G)$ (defined in Section~\ref{sec:background}).
Our purpose in this paper is to study the analogue of the edge polytope obtained by taking the convex hull of the columns of $L(G)$, resulting in a lattice simplex that we call the \emph{Laplacian simplex} of $G$ and denote $T_G$.

While to our knowledge the simplex $T_G$ has not been previously studied, there has been recent research regarding graph Laplacians from the perspective of polyhedral combinatorics and integer-point enumeration.
For example, M. Beck and the first author investigated hyperplane arrangements defined by graph Laplacians with connections to nowhere-harmonic colorings and inside-out polytopes~\cite{BeckBraunNHColorings}.
A. Padrol and J. Pfeifle explored Laplacian Eigenpolytopes~\cite{padrolpfeiflelaplacian} with a focus on the effect of graph operations on the associated polytopes.
The first author, R. Davis, J. Doering, A. Harrison, J. Noll, and C. Taylor studied integer-point enumeration for polyhedral cones constrained by graph Laplacian minors~\cite{BraunLaplacianMinors}.
In a recent preprint~\cite{matrixtreetheorem}, A. Dall and J. Pfeifle analyzed polyhedral decompositions of the zonotope defined as the Minkowski sum of the line segments from the origin to each column of $L(G)$ in order to give a polyhedral proof of the Matrix-Tree Theorem.

Beyond the motivation of studying $T_G$ in order to develop a Laplacian analogue of the theory of edge polytopes, our primary motivation in this paper is the following conjecture (all undefined terms are defined in Section~\ref{sec:background}).

\begin{conjecture}[Hibi and Ohsugi \cite{hibiohsugiconj}] \label{conj:hibiohsugi}
If $\mathcal P$ is a lattice polytope that is reflexive and satisfies the integer decomposition property, then $\mathcal P$ has a unimodal Ehrhart $h^*$-vector.
\end{conjecture}

The cause of unimodality for $h^*$-vectors in Ehrhart theory is mysterious.
Schepers and van Langenhoven~\cite{schepersvanl} have raised the question of whether or not the integer decomposition property alone is sufficient to force unimodality of the $h^*$-vector for a lattice polytope.
In general, the interplay of the qualities of a lattice polytope being reflexive, satisfying the integer decomposition property, and having a unimodal $h^*$-vector is not well-understood~\cite{BraunUnimodal}.
Thus, when new families of lattice polytopes are introduced, it is of interest to explore how these three properties behave for that family.
Further, lattice simplices have been shown to be a rich source of examples and have been the subject of several recent investigations, especially in the context of Conjecture~\ref{conj:hibiohsugi}~\cite{BraunDavisFreeSum,BraunDavisSolusIDP,PayneLattice,SolusNumeral}.

\subsection{Our Contributions}

After reviewing necessary background in Section~\ref{sec:background}, we introduce and establish basic properties of Laplacian simplices in Section~\ref{sec:lapsimp}.
We show that several graph-theoretic operations produce reflexive Laplacian simplices (Theorem~\ref{thm:bridge} and Proposition~\ref{evenreflexive}).
We prove that if $G$ is a tree, odd cycle, complete graph, or the whiskering of an even cycle, then $T_G$ is reflexive (Proposition~\ref{prop:trees}, Theorem~\ref{cycle}, Proposition~\ref{evenreflexive}, and Theorem~\ref{complete}).
As a result of a general investigation of the structure of $h^*$-vectors for odd cycles (Theorem~\ref{primes}), we show that if $n$ is odd then $T_{C_n}$ does not have the integer decomposition property (Corollary~\ref{cor:oddcyclenotidp}).
On the other hand, we show that $T_{K_n}$ does have the integer decomposition property since it admits a regular unimodular triangulation (Corollary~\ref{cor:completeidp}).
We prove that for trees, odd cycles, and complete graphs, the $h^*$-vectors of their Laplacian simplices are unimodal (Corollary~\ref{cor:treeunim}, Theorem~\ref{unimodal}, and Corollary~\ref{cor:completeunimodal}).
Additionally, we provide a combinatorial interpretation of the $h^*$-vector for $T_{K_n}$ (Proposition~\ref{prop:completeh*}) and we determine that, when $n$ is an odd prime, the $h^*$-vector of $T_{C_n}$ is given by $(h_0^*,\ldots,h_{n-1}^*)=(1,\ldots,1,n^2-n+1,1,\ldots, 1)$ (Theorem~\ref{primes}).

\section{Background}\label{sec:background}

\subsection{Reflexive Polytopes}\label{sec:reflexive}
A \emph{lattice polytope} of dimension $d$ is the convex hull of finitely many points in $\Z^n$, which together affinely span a $d$-dimensional hyperplane of $\R^n$. 
Two lattice polytopes are \emph{unimodularly equivalent} if there is a lattice preserving affine isomorphism mapping them onto each other. 
Consequently we consider lattice polytopes up to affine automorphisms of the lattice.
The \emph{dual polytope} of a full dimensional polytope $\mathcal P$ which contains the origin in its interior is 
\[
\mathcal P^* := \{ x \in \R^d \mid x \cdot y \le 1 \text{ for all } y \in \mathcal P\} \, .
\]
Duality satisfies $(\mathcal P^*)^* = \mathcal P$.
A $d$-polytope formed by the convex hull of $d+1$ vertices is called a \emph{$d$-simplex}.

Reflexive polytopes are a particularly important class of polytopes first introduced in \cite{BatyrevDualPolyhedra}.
\begin{definition} 
A lattice polytope $\mathcal P$ is called \emph{reflexive} if it contains the origin in its interior, and its dual $\mathcal P^*$ is a lattice polytope. 
\end{definition}
Any lattice translate of a reflexive polytope is also called reflexive.
The following generalization of reflexive polytopes was introduced in \cite{Nill}.
A lattice point is \emph{primitive} if the line segment joining it and the origin contains no other lattice points.
The local index $\ell_F$ is equal to the integral distance from the origin to the affine hyperplane spanned by $F$. 

\begin{definition} \label{def:ref}
A lattice polytope $\mathcal P$ is \emph{$\ell$-reflexive} if, for some $\ell \in \Z_{>0}$, the following conditions hold:
\begin{enumerate}[(i)]
\item $\mathcal P$ contains the origin in its (strict) interior;
\item The vertices of $\mathcal P$ are primitive;
\item For any facet $F$ of $\mathcal P$ the local index $\ell_F = \ell$.
\end{enumerate}
\end{definition}
We refer to $\mathcal P$ as a \emph{reflexive polytope of index $\ell$}.
The reflexive polytopes of index $1$ are precisely the reflexive polytopes in Definition~\ref{def:ref}.  

\subsection{Ehrhart Theory}\label{sec:ehrhart}
For $t \in \Z_{>0}$, the \emph{$t$\ss{th} dilate} of $P$ is given by $t \mathcal P := \{tp \mid p \in \mathcal P\}$.
One technique used to recover dilates of polytopes is \emph{coning over the polytope}.
Given $\mathcal P = \conv{v_1, \ldots , v_m} \subseteq \R^n$, we lift these vertices into $\R^{n+1}$ by appending $1$ as their last coordinate to define $w_1=(v_1, 1), \ldots , w_m = (v_m, 1).$
The \emph{cone over $\mathcal P$} is
\[
\text{cone}(\mathcal P) = \{ \lambda_1 w_1 + \lambda_2 w_2 + \cdots + \lambda_m w_m \mid \lambda_1, \lambda_2, \ldots , \lambda_m \ge 0\} \subseteq \R^{n+1} \, . 
\]
For each $t \in \Z_{>0}$ we recover $t\mathcal P$ by considering $\text{cone}(\mathcal P) \cap \{z_{n+1}=t\}$.
To record the number of lattice points we let $L_{\mathcal P}(t) = | t \mathcal P \cap \Z^n |$.
In \cite{Ehrhart}, Ehrhart proved that $L_{\mathcal P}(t)$, called the Ehrhart polynomial of $\mathcal P$, is a polynomial in degree $d=\dim(\mathcal{P})$ with generating function 
\[
\text{Ehr}_{\mathcal P}(z) = 1 + \sum_{t\ge1} L_{\mathcal P}(t)z^t = \frac{h_d^*z^d + h_{d-1}^*z^{d-1} + \cdots + h_1^*z + h_0^*}{(1-z)^{d+1}}.
\]
The above is referred to as the \emph{Ehrhart series} of $\mathcal P$. 
We call $h^*(\mathcal P) = (h_0^*, h_1^*, \ldots , h_d^*)$ the \emph{$h^*$-vector} or \emph{$\delta$-vector} of $\mathcal P$. 
The \emph{Euclidean volume} of a polytope $\mathcal P$ is vol$(\mathcal P) = \frac{1}{d!}\sum_{i=0}^d h_i^*$. 
The \emph{normalized volume} is given by $d!\text{vol}(\mathcal P) = \sum_{i=0}^d h_i^*$.
Stanley proved the $h^*$-vector of a convex lattice $d$-polytope satisfies $h_0^* =1$ and $h_i^* \in \Z_{\ge 0}$ \cite{Stanley1}. 
Note that if $\mathcal P$ and $\mathcal Q$ are lattice polytopes such that $\mathcal Q$ is the image of $\mathcal P$ under an affine unimodular transformation, then their Ehrhart series are equal.

A vector $x = (x_0, x_1, \ldots , x_d)$ is \emph{unimodal} if there exists a $j \in [d]$ such that $x_i \le x_{i+1}$ for all $0 \le i < j$ and $x_k\ge x_{k+1}$ for all $j \le k < d$.
A major open problem in Ehrhart theory is to determine properties of $\mathcal P$ that imply unimodality of $h^*(\mathcal P)$ \cite{BraunUnimodal}.
For the case of symmetric $h^*$-vectors, Hibi established the following connection to reflexive polytopes.
\begin{theorem}[Hibi \cite{Hibi1}]
A $d$-dimensional lattice polytope $\mathcal P \subseteq \R^d$ containing the origin in its interior is reflexive if and only if $h^*(\mathcal P)$ satisfies $h_i^* = h_{d-i}^*$ for $0 \le i \le \lfloor \frac{d}{2} \rfloor.$
\end{theorem}
Thus, when investigating symmetric $h^*$-vectors, reflexive polytopes (and, more generally, Gorenstein polytopes) are the correct class to work with.
As indicated by Conjecture~\ref{conj:hibiohsugi}, the following property has been frequently correlated with unimodality, and is interesting in its own right.

\begin{definition}
A lattice polytope $\mathcal P \subseteq \R^n$ has the \emph{integer decomposition property} if, for every integer $t \in \Z_{>0}$ and for all $p \in t\mathcal P \cap \Z^n$, there exists $p_1 , \ldots , p_t \in \mathcal P \cap \Z$ such that $p = p_1 + \cdots + p_t$.
We will frequently say that $\mathcal P$ is IDP when $\mathcal P$ possesses this property.
\end{definition}

It is well-known that if $\mathcal P$ admits a unimodular triangulation, then $\mathcal P$ is IDP; we will use this fact when analyzing complete graphs.

\subsection{Lattice Simplices}

Simplices play a special role in Ehrhart theory, as there is a method for computing their $h^*$-vectors that is simple to state (though not always to apply).

\begin{definition} \label{def:fpp}
Given a lattice simplex $\mathcal{P}\subset \R^{n-1}$ with vertices $\{v_i\}_{i \in [n]}$, the \emph{fundamental parallelepiped} of $\mathcal P$ is the subset of $\cone{\mathcal P}$ defined by 
\[
\Pi_{\mathcal P} := \left\{ \sum_{i=1}^n \lambda_i (v_i, 1) \mid 0 \le \lambda_i < 1  \right\} \, .
\]
Further, $|\Pi_{\mathcal{P}}\cap \Z^n|$ is equal to the determinant of the matrix whose $i$\ss{th} row is given by $(v_i,1)$.
\end{definition}

\begin{lemma}[see Chapter 3 of \cite{BeckRobinsCCD}]\label{lem:fpp}
Given a lattice simplex $\mathcal P$, 
\[
h_i^*(\mathcal{P}) = \left\lvert \Pi_{\mathcal P} \cap \{ x \in \Z^n \mid x_{n}=i \} \right\rvert \, .
\]
\end{lemma}

Using the notation from Definition~\ref{def:fpp}, let $A$ be the matrix whose $i$\ss{th} row is $(v_i,1)$.
One approach to determine $h^*(\mathcal{P})$ in this case is to recognize that finding lattice points in $\Pi_{\mathcal{P}}$ is equivalent to finding integer vectors of the form $\lambda \cdot A$ with $0\leq \lambda_i< 1$ for all $i$.
Cramer's rule implies the $\lambda \in \Q^n$ that yield integer vectors will have entries of the form
\[
\lambda_i = \frac{b_i}{\det{A}} < 1
\] 
for $b_i \in \Z_{\geq 0}$.
In particular, if $x = \frac{1}{\det(A)} b \cdot A \in \mathbb{Z}^n$, then $b_i = \det{A(i, x)}$ where $A(i, x)$ is the matrix obtained by replacing the $i$\ss{th} row of $A$ by $x$. 
Since $A(i, x)$ is an integer matrix, $\det{A(i,x)} \in \Z$. Notice that for any $\lambda$, the last coordinate of $\lambda A$ is $\langle \lambda, \1 \rangle = \sum_{i=1}^n \frac{b_i}{\det{A}}.$ 
Thus, we have
\[ 
\Pi_{\mathcal{P}} \cap \Z^n  =  \Z^n\cap \left\{ \frac{1}{\det{A}} b \cdot A \mid 0 \le b_i < \det(A), b_i\in \Z, \sum_{i=1}^n b_i \equiv 0 \bmod \det(A) \right\} \, . 
\]
One profitable method for determining the lattice points in $\Pi_{\mathcal{P}}$ is to find the $\det(A)$-many lattice points in the right-hand set above, by first considering all the $b$-vectors that satisfy the given modular equation.

\subsection{Graph Laplacians}\label{sec:laplacians}
Let $G$ be a connected graph with vertex set $V(G) = [n] := \{1, 2, \ldots, n\}$ and edge set $E(G)$.
The \emph{Laplacian matrix} $L$ of a graph $G$ is defined to be the difference of the degree matrix and the $\{0, 1\}$-adjacency matrix of a graph.
Thus, $L$ has rows and columns indexed by $[n]$ with entries $a_{ii}=\deg{i}$, $a_{ij} = -1$ if $\{i,j\} \in E(G)$, and $0$ otherwise.
We let $\kappa$ denote the number of spanning trees of $G$.
The following facts are well-known \cite{Bapat}.

\begin{prop} 
The Laplacian matrix $L$ of a connected graph $G$ with vertex set $[n]$ satisfies the following:
\begin{enumerate}[(i)]
\item $L \in \Z^{n \times n}$ is symmetric.
\item Each row and column sum of $L$ is $0$. 
\item $\ker_{\R}{L} = \langle \1 \rangle$ and $\text{im}_{\R}{\text{ }L}=\langle \1 \rangle^{\perp}$
\item $\rk{L} = n-1$
\item (The Matrix-Tree Theorem \cite{Kirchhoff}) Any cofactor of $L$ is equal to $\kappa$.
\end{enumerate}
\end{prop}

In this paper we often refer to a submatrix of $L$ defined by restricting to specified rows and columns.
For $S, T \subseteq [n]$, define $L(S \mid T)$ to be the matrix with rows from $L$ indexed by $[n]\setminus S$ and columns from $L$ indexed by $[n] \setminus T$.  
Equivalently, $L(S \mid T)$ is obtained from $L$ by the deletion of rows indexed by $S$ and columns indexed by $T$.
For simplicity, we define $L(i)$ to be the matrix obtained by deleting the $i$\ss{th} column of $L$, that is, $L(i) := L( \emptyset \mid i) \in \Z^{n \times (n-1)}$.

\section{The Laplacian Simplex of a Finite Graph}\label{sec:lapsimp}

\subsection{Definition and Basic Properties}
Assume that $G$ is a connected graph with Laplacian matrix $L$.
Consider $L(i) \in \Z^{n \times (n-1)}$.
It is a straightforward exercise to show the rank of $L(i)$ is $n-1$.
We recognize the rows of $L(i)$ as points in $\Z^{n-1}$ and consider their convex hull, $\conv{L(i)^T}$, where $\conv{A}$ refers to the convex hull of the columns of the matrix $A$.
Notice the rows of ${L(i)}$ form a collection of $n$ affinely independent lattice points, which makes $\conv{L(i)^T}$ an $n-1$ dimensional simplex.

\begin{prop}\label{equivalence} 
The lattice simplices $\conv{L(i)^T}$ and $\conv{L(j)^T}$ are unimodularly equivalent for all $i, j \in [n]$.
\end{prop}

\begin{proof} 
Notice the matrices $L(i)$ and $L(j)$ differ by only one column when $i \ne j$.
In particular we can write $L(i) \cdot U = L(j)$ where $U \in \Z^{n-1 \times n-1}$ has columns $c_k$ for $1 \le k \le n-1$ defined to be
\[ 
c_k = \left\{
\begin{array}{ll}
    e_{\ell} & \text{column $k$ in $L(j)$ is column $\ell$ in $L(i)$} \\
    (-1, -1, \ldots, -1)^T & \text{column $k$ in $L(j)$ is not among columns of $L(i)$ }
\end{array} \right\}
\] 
where $e_{\ell}$ is the vector with a $1$ in the $\ell\ss{th}$ entry and $0$ else.

Notice $U$ has integer entries and $\det{U} = \pm 1$, as computed by expanding along the column with all entries equal to $-1$. 
This shows $U$ is a unimodular matrix.
Further, $U$ maps the vertices of $\conv{L(i)^T}$ onto the vertices of $\conv{L(i)^T}$.
Thus $\conv{L(i)^T}$ and $\conv{L(j)^T}$ are unimodularly equivalent lattice polytopes.
\end{proof}

Given a fixed graph $G$, we choose a representative for this equivalence class of lattice simplices to be used throughout, unless otherwise noted.
Let $B = \{ e_1 - e_2, e_2-e_3, \ldots, e_{n-1} - e_n\}$ be the standard basis for the orthogonal complement of the all-ones vector $\1 \in \R^n$, where $e_i \in \R^n$ is the standard basis vector that contains a $1$ in the $i$\ss{th} entry and $0$ else.
Then $B$ is a basis of the column space of $L$.
Define $L_B \in \Z^{n \times (n-1)}$ to be the representation of the matrix $L$ with respect to the basis $B$.
In practice, $L_B$ can be computed using the matrix multiplication $L_B = L \cdot A$ where $A$ is the upper triangular $(n \times (n-1))$ matrix with entries
\begin{equation}\label{eqn:A}
a_{ij}= \left\{
\begin{array}{ll}
    1 & i \le j \le
     n-1 \\
    0 &  \text{else}
\end{array} \right\}.
\end{equation}

\begin{example}
Given the cycle $C_5$ of length five, we have
\[
L = 
\left[
\begin{array}{rrrrr}
2  & -1 & 0  & 0 & -1 \\
    -1 & 2 & -1 & 0  & 0 \\
    0  & -1& 2 & -1  & 0 \\
    0  & 0 &-1 & 2 & -1 \\
    -1 & 0& 0 & -1 & 2
\end{array}
\right] 
\hspace{.2in}
L_B = 
\left[
\begin{array}{rrrr}
    2 & 1 &  1&  1 \\
    -1& 1 &  0  & 0 \\
    0 & -1&   1  & 0 \\
    0 &  0  & -1 & 1 \\
    -1 & -1 &  -1 & -2
\end{array}
\right] \, .
\]
\end{example}

This brings us to the object of study in this paper. 

\begin{definition} 
For a connected graph $G$, the $n-1$ dimensional lattice simplex \[T_G:= \conv{(L_B)^T} \subseteq \R^{n-1}\] is called the $\emph{Laplacian Simplex}$ associated to the graph $G$.
\end{definition}

\begin{prop}\label{properties}
Let $G$ be a connected graph on $n$ vertices.
\begin{enumerate}[(i)]
\item $T_G$ is a representative of the equivalence class $\{\conv{L(i)^T}\}_{i \in [n]}$.
\item $T_G$ has normalized volume equal to $n \cdot \kappa$.
\item $T_G$ contains the origin in its interior.
\item $h_i^*(T_G) \ge 1$ for all $0 \le i \le n-1.$
\end{enumerate}
\end{prop}

\begin{proof}
\begin{enumerate}[(i)]
\item Notice we can write $L(n) \cdot A(n \mid \emptyset) = L_B$ where $A$ is the matrix defined in equation (\ref{eqn:A}). 
Let $U:=A(n \mid \emptyset)$.  
Then $U$ is the upper diagonal matrix of all ones so that $\det{U} = 1$.
This implies $T_G$ is unimodularly equivalent to $\conv{L(n)^T}$.
By Proposition \ref{equivalence}, the result follows.
\item Since $T_G$ is a simplex, the normalized volume of $T_G$ is equal to 
\[
\left| \det{[L_B\mid\1]} \right|= \left| \sum_{i=1}^n (-1)^{i+n} M_{in} \right| = \left| \sum_{i=1}^n C_{in} \right|,
\] 
where $M_{i,n}$ is a minor of $[L_B \mid \1]$, $C_{i,n}$ is the corresponding cofactor, and the determinant is expanded along the appended column of ones.
The relation $L(n) \cdot U = L_B$ yields $L(i \mid n) \cdot U = L_B(i \mid \emptyset).$
Then for each cofactor,
\begin{equation*}
\begin{split}
C_{i,n} &= (-1)^{i+n}\det L_B(i \mid \emptyset ) \\
&= (-1)^{i+n}\det(L(i\mid n)\cdot U) \\
&= (-1)^{i+n}\det L(i\mid n)\det U \\
&= (-1)^{i+n}\det L(i\mid n) \\
&= \bar{C}_{i,n} \\
&= \kappa 
\end{split}
\end{equation*}
where $\bar{C}_{i,n}$ is the cofactor of $L$, and the last equality is a result of the Matrix Tree Theorem.
Summing over all $i \in [n]$ yields the desired result.

\item Note the sum of all rows of $L_B$ is $0$, and $L_B$ has no column with all entries equal to $0$. It follows that $(0, \ldots, 0) \in \Z^{n}$ is in the interior of $T_G$.   

\item Observe each column in $L_B$ sums to $0$.
Consider lattice points of the form 
\[
p_i = \left(\frac{i}{n}, \frac{i}{n}, \ldots, \frac{i}{n}\right) \cdot [L_B \mid \1] = \left(0, 0, \ldots, 0, i  \right) \in \Z^{1 \times n} 
\] 
for each $0 \le i < n$. 
Then $p_i \in \Pi_{T_G} \cap \{x \in \Z^n \mid x_{n} = i\}$ implies $h_i^*(T_G) \ge 1$ for each $0 \le i \le n-1$.
\end{enumerate}
\end{proof}

\begin{example}
The simplex $T_{C_5}$ is obtained as the convex hull of the columns of the transpose of
\[
L_B = 
\left[
\begin{array}{rrrr}
    2 & 1 &  1&  1 \\
    -1& 1 &  0  & 0 \\
    0 & -1&   1  & 0 \\
    0 &  0  & -1 & 1 \\
    -1 & -1 &  -1 & -2
\end{array}
\right] \, .
\]
The determinant of $L_B$ with a column of ones appended is easily computed to be $25$.
By applying Lemma~\ref{lem:fpp} to $T_{C_5}$, it is straightforward to verify that $h^*(T_{C_5})=(1,1,21,1,1)$.
\end{example}

In the proof of (ii) in Proposition \ref{properties} above, we showed the minor obtained by deleting the $i$\ss{th} row of $L_B$ is equal to the minor obtained by deleting the $n$\ss{th} column and the $i$\ss{th} row of $L$ for some $i \in [n]$, i.e., $\det{L_B(i \mid \emptyset)} = \det{L(i \mid n)}$ for any $i \in [n]$.  
The second minors of $L_B$ and $L$ are related in the following manner, which we will need in subsequent sections.

\begin{lemma}\label{determinant} 
Let $i,k \in [n]$ and $j \in [n-1]$ such that $i \ne k$.  
Then 
\[
\det{L_B(i,k \mid j )} = \det{L(i,k \mid j,n)} + \det{L(i,k \mid j+1, n)}.
\] 
In the case $j=n-1$, $\det{L_B(i,k \mid n-1 )} = \det{L(i,k \mid n-1,n)}.$
\end{lemma}
\begin{proof}
Recall $L_B = L \cdot A$ where $A$ is the $n \times (n-1)$ upper diagonal matrix defined in equation (\ref{eqn:A}).
It follows $L_B(i, k \mid j) = L(i, k \mid \emptyset) \cdot A(j)$.
Apply the Cauchy-Binet formula to compute the determinant 
\begin{equation*}
\begin{split}
\det{L_B(i,k \mid j)} &= \sum_{S \in \binom{[n]}{n-2}} \det{L(i, k \mid \emptyset)_{[n-2],S}} \det{A(j)_{S,[n-2]}} \\
&= \det{L(i, k \mid \emptyset)_{[n-2],[n] \setminus \{j,n\}}} \det{A(j)_{[n]\setminus \{j,n\},[n-2]}}  \\ 
& \phantom{.......} + \det{L(i, k \mid \emptyset)_{[n-2],[n]\setminus \{(j+1),n\}}} \det{A(j)_{[n]\setminus \{(j+1),n\},[n-2]}} \\
&= \det{L(i,k \mid j,n)} + \det{L(i,k \mid j+1, n)}.
\end{split}
\end{equation*}


The only nonzero terms in the sum arise from choosing $(n-2)$ linearly independent rows in $A$.
Based on the structure of $A$, there are only two ways to do this unless we are in the case $j=n-1$ in which there is exactly one way.
\end{proof}

The following is a special case of a general characterization of reflexive simplices using cofactor expansions. 

\begin{theorem}\label{characterization}
For a connected graph $G$ with Laplacian matrix $L$, $T_G$ is reflexive if and only if for each $i \in [n]$, $\kappa$ divides 
\[
\sum_{k=1}^{n-1} C_{kj} = \sum_{k=1}^{n-1} (-1)^{k+j}M_{kj}
\] 
for each $1 \le j \le n-1$.
Here $C_{kj}$ is the cofactor and $M_{kj}$ is the minor of the matrix $L_B(i \mid \emptyset) \in \Z^{(n-1) \times (n-1)}$.
\end{theorem}

\begin{proof}
We show $T_{G}$ is reflexive by showing the vertices of its dual polytope are lattice points. 
By \cite[Theorem 2.11]{ZieglerLectures}, the hyperplance description of the dual polytope is given by $T_{G}^* = \{x \in \R^{n-1} \mid L_B \cdot x \le \1 \}.$
Each intersection of $(n-1)$ hyperplanes will yield a unique vertex of $T_{G}^*$ since any first minor of $L_B$ is nonzero. 
Let $\{v_1, v_2, \ldots, v_{n} \}$ be the set of vertices of $T^*_{G}$.
Each $v_i$ satisfies
\[
L_B(i \mid \emptyset) v_i = \1
\]
for $i \in [n]$. 
Reindex the rows of $L_B(i \mid \emptyset)$ in increasing order by $[n-1]$. 
We can write 
\[
v_i = L_B(i \mid \emptyset)^{-1} \cdot \1
	= \frac{1}{\det{L_B(i \mid \emptyset)}} C^T \cdot \1 
\]
where $C^T$ is the $(n-1) \times (n-1)$ matrix whose whose $(j,k)$ entry is the $(k,j)$ cofactor of $L_B( i \mid \emptyset)$, which we denote as $C_{k j}$.
Since $\det{L_B(i \mid \emptyset)} = \det{L( i \mid n)} = \pm \kappa$, each vertex is of the form 
\[
v_i = \frac{1}{\pm \kappa} \left( \sum_{k=1}^{n-1} C_{k1}, \sum_{k=1}^{n-1} C_{k2}, \ldots, \sum_{k=1}^{n-1} C_{k(n-1)}\right)^T,
\]
which is a lattice point if and only if $\kappa$ divides each coordinate.
\end{proof}

\begin{remark}
Apply Lemma \ref{determinant} to Proposition \ref{characterization} to yield a condition on the second minors of $L$ when determining if $T_G$ is reflexive.
Notice
\begin{equation*}
\begin{split}
(C^T)_{j k} &= C_{k j} \\
&= (-1)^{k + j} \det{L_B( i, k \mid j )} \\
&= (-1)^{k + j} \left(\det{L(i, k \mid j, n)}+\det{L(i,k \mid j+1, n)} \right), 
\end{split}
\end{equation*}
which shows for a given $v_i$, its $\ell \ss{th}$ coordinate has the form 
\[
  \frac{1}{\pm \kappa} \sum_{k=1}^{n-1}C_{k \ell} = \frac{1}{\pm \kappa} \sum_{k=1}^{n-1} (-1)^{k + \ell} \left(\det{L(i, k \mid \ell, n)}+\det{L(i,k \mid \ell+1, n)} \right).
\]
\end{remark}

\begin{remark}
Computing alternating sums of second minors of Laplacian matrices can be challenging.
Thus, we often verify reflexivity by explicitly computing the vertices of $T_G^*$ via ad hoc methods.
\end{remark}

\subsection{Graph Operations and Laplacian Simplices}

We next introduce an operation on a graph that preserves the lattice-equivalence class of $T_G$.

\begin{prop}\label{algorithm}
Let $G$ be a connected graph on $n$ vertices such that the following cut is possible.
Partition $V(G)$ into vertex sets $A$ and $B$ such that all edges between $A$ and $B$ are incident to a single vertex $x \in A$; label those edges $\{e_1, \ldots, e_k \}$. 
Additionally suppose $x$ has a leaf with adjacent vertex $y \in A$.
Form a new graph $G'$ by moving the edges $\{e_1, \ldots, e_k\}$ previously incident to $x$ to  be incident to $y$. 
Then $G'$ has vertex set $V(G)$, and edge set $\left( E(G) \setminus \{e_1, \ldots, e_k\}\right) \cup \{\{y, v_i \}:i=1,\ldots,k\}$ where $e_i = \{x, v_i\} \in E(G)$. 
Then $T_{G} \cong T_{G'}$.
\end{prop}

\begin{proof}
Label the vertices of $G$ with $[n]$.
Observe $G'$ has the same labels since $V(G) = V(G')$.
We refer to each vertex by its label for simplicity.
Let $N_G(i)$ be the set of neighbors of vertex $i$ in $G$, that is, $N_G(i) := \{ j \in V(G) \mid \{i, j\} \in E(G) \}$.
Let $L$ be the Laplacian matrix of $G$ and $L'$ be the Laplacian matrix of $G'$.
We describe row operations that take each row $r_i \in L$ to row $r'_i \in L'$.
For each $i \in V(G)$, $1 \le i \le n$, we have the following cases.

Consider $i \in A$ such that $i \ne x, y$.
Then $N_G(i) = N_{G'}(i)$, so we set $r'_i=r_i$ since the $i\ss{th}$ row is the same in $L$ and $L'$. Then $r'_i \in L'$.

Consider $i \in B \setminus N_{G'}(x)$.
Again, $N_G(i) = N_{G'}(i)$, so we set $r'_i = r_i$ and have $r'_i \in L'$. 

Consider $i \in B \cap N_{G'}(x)$.
The degree of $i$ is constant in $G$ and $G'$, but $\{i, x\} \in E(G)$ becomes $\{i, y\} \in E(G')$ in the described algorithm.
Set $r'_i = r_i - r_y$ to reflect the change in incident edges of $i$ from $G$ to $G'$.
Since $y \in V(G)$ is a leaf, $r'_i$ now has a $0$ in the $x\ss{th}$ coordinate, a $-1$ in the $y\ss{th}$ coordinate, and all remaining coordinates are unchanged. Then $r'_i \in L'.$

Consider $i = x$. 
Set $r'_x = r_x + \sum_{j \in B} r_j$.
Observe $N_G(x) \setminus N_{G'}(x) = \{v_1, \ldots, v_k \}$.
Then adding $\sum_{\ell =1}^{k} r_{v_\ell}$ decreases the $x\ss{th}$ coordinate of $r_x$ by $k$, which is the new degree of vertex $x \in V(G')$.
Adding the other rows does contribute to the $x\ss{th}$ coordinate of $r'_x$ since those vertices are not adjacent to $x \in V(G)$; however, we must add all rows corresponding to $j \in B$ to obtain a $0$ in all coordinates indexed by $j \in B$.
Notice the coordinates indexed by the vertices in $A$ remain fixed.
Then $r'_x \in L'$.

Finally consider $i=y$.
Set $r'_y = (k+1)r_y -\sum_{j \in B} r_j$.
The $y\ss{th}$ coordinate of $r'_y$ is $k+1$, which is the degree of $y$ in $V(G')$.
Observe $N_{G'}(y) \setminus N_G(y) = \{v_1, \ldots, v_k\}.$ 
Then subtracting $\sum_{\ell =1}^k r_{v_\ell}$ from $(k+1)r_y$ ensures the $x\ss{th}$ coordinate of $r'_y$ is $-1$. 
We subtract all rows corresponding to $j \in B$ from $(k+1)r_y$ to obtain a $-1$ in all coordinates of $r'_y$ indexed by $\{v_\ell \}_{\ell=1}^k$.
Then $r'_y \in L'$.

It is straightforward to verify that the collection of row operations described above is a unimodular transformation of the Laplacian matrix and thus can be represented by the multiplication of unimodular matrix $U \in \Z^{n \times n}$ such that $U\cdot L = L'$.
It follows that $U\cdot L(n) = L'(n)$.
Thus  $\conv{L(n)^T}= \conv{L'(n)^T}$, and we have shown $T_G \cong T_{G'}$. 
\end{proof}

\begin{example}\label{5} 
In the figure below, the graph on the left is the wedge of $K_5$ and $C_5$ with a leaf, and the graph on the right is the bridge of $K_5$ and $C_5$ with the appropriate labels.
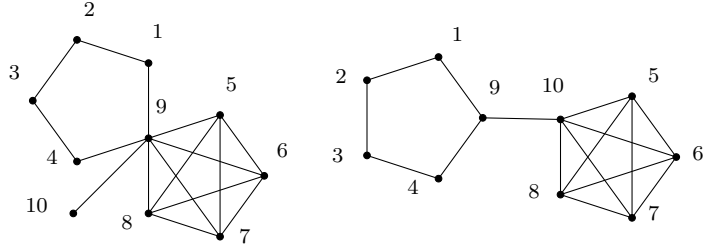
\definecolor{zzttqq}{rgb}{0.6,0.2,0.}
\begin{center}
\begin{tikzpicture}[scale=.5][line cap=round,line join=round,>=triangle 45,x=1.0cm,y=1.0cm]
\clip(7.5,0.5) rectangle (16.2,8.);
\fill[color=zzttqq] (10.,-4.) -- (10.,-6.) -- (11.902113032590307,-6.618033988749895) -- (13.077683537175254,-5.) -- (11.902113032590307,-3.381966011250105) -- cycle;
\fill(4.854972311947651,-2.95708556678168) -- (4.854972311947651,-4.95708556678168) -- (6.757085344537957,-5.575119555531575) -- (7.932655849122904,-3.957085566781681) -- (6.757085344537959,-2.3390515780317855) -- cycle;
\fill[color=white](12.,4.) -- (12.,2.) -- (13.902113032590307,1.381966011250105) -- (15.077683537175254,3.) -- (13.902113032590307,4.618033988749895) -- cycle;
\draw (10.,-4.)-- (10.,-6.);
\draw (10.,-6.)-- (11.902113032590307,-6.618033988749895);
\draw (11.902113032590307,-6.618033988749895)-- (13.077683537175254,-5.);
\draw (13.077683537175254,-5.)-- (11.902113032590307,-3.381966011250105);
\draw (11.902113032590307,-3.381966011250105)-- (10.,-4.);
\draw (4.854972311947651,-2.95708556678168)-- (4.854972311947651,-4.95708556678168);
\draw (4.854972311947651,-4.95708556678168)-- (6.757085344537957,-5.575119555531575);
\draw (6.757085344537957,-5.575119555531575)-- (7.932655849122904,-3.957085566781681);
\draw (7.932655849122904,-3.957085566781681)-- (6.757085344537959,-2.3390515780317855);
\draw (6.757085344537959,-2.3390515780317855)-- (4.854972311947651,-2.95708556678168);
\draw (12.,4.)-- (12.,2.);
\draw (12.,2.)-- (13.902113032590307,1.381966011250105);
\draw (13.902113032590307,1.381966011250105)-- (15.077683537175254,3.);
\draw (15.077683537175254,3.)-- (13.902113032590307,4.618033988749895);
\draw (13.902113032590307,4.618033988749895)-- (12.,4.);
\draw (10.,2.)-- (12.,4.);
\draw (7.932655849122904,-3.957085566781681)-- (10.,-4.);
\draw (13.902113032590307,4.618033988749895)-- (12.,2.);
\draw (15.077683537175254,3.)-- (12.,4.);
\draw (13.902113032590307,4.618033988749895)-- (13.902113032590307,1.381966011250105);
\draw (15.077683537175254,3.)-- (12.,2.);
\draw (13.902113032590307,1.381966011250105)-- (12.,4.);
\draw (10.,-4.)-- (13.077683537175254,-5.);
\draw (10.,-4.)-- (11.902113032590307,-6.618033988749895);
\draw (11.902113032590307,-3.381966011250105)-- (10.,-6.);
\draw (11.902113032590307,-3.381966011250105)-- (11.902113032590307,-6.618033988749895);
\draw (10.,-6.)-- (13.077683537175254,-5.);
\draw (10.097886967409693,6.618033988749895)-- (8.922316462824746,5.);
\draw (8.922316462824746,5.)-- (10.097886967409693,3.381966011250105);
\draw (10.097886967409693,3.381966011250105)-- (12.,4.);
\draw (12.,6.)-- (10.097886967409693,6.618033988749895);
\draw (12.,4.)-- (12.,6.);
\begin{scriptsize}
\draw [fill=black] (10.,-4.) circle (2.5pt);
\draw[color=black] (9.79819918683186,-3.087469617134429) node {10};
\draw [fill=black] (10.,-6.) circle (2.5pt);
\draw[color=black] (9.29899178703002,-5.9011840523811605) node {8};
\draw [fill=black] (11.902113032590307,-6.618033988749895) circle (2.5pt);
\draw[color=black] (12.47576614940536,-6.491156433965153) node {7};
\draw [fill=black] (13.077683537175254,-5.) circle (2.5pt);
\draw[color=black] (13.655710912573344,-4.902769252777482) node {6};
\draw [fill=black] (11.902113032590307,-3.381966011250105) circle (2.5pt);
\draw[color=black] (12.47576614940536,-2.8151746717879713) node {5};
\draw [fill=black] (4.854972311947651,-2.95708556678168) circle (2.5pt);
\draw[color=black] (4.170770316338401,-2.8605571626790476) node {2};
\draw [fill=black] (4.854972311947651,-4.95708556678168) circle (2.5pt);
\draw[color=black] (4.080005334556248,-4.948151743668558) node {3};
\draw[color=black] (7.256779696931588,-1.7259948904021396) node {1};
\draw [fill=black] (6.757085344537957,-5.575119555531575) circle (2.5pt);
\draw[color=black] (6.076834933763605,-5.810419070599008) node {4};
\draw [fill=black] (7.932655849122904,-3.957085566781681) circle (2.5pt);
\draw[color=black] (8.255194496535267,-3.132852108025505) node {9};
\draw [fill=black] (6.757085344537959,-2.3390515780317855) circle (2.5pt);
\draw [fill=black] (12.,4.) circle (2.5pt);
\draw[color=black] (12.339618676732131,4.854466288803926) node {$9$};
\draw [fill=black] (12.,2.) circle (2.5pt);
\draw[color=black] (11.431968858910606,1.7684569082107366) node {$8$};
\draw [fill=black] (13.902113032590307,1.381966011250105) circle (2.5pt);
\draw[color=black] (14.56336073039487,1.405396981082126) node {$7$};
\draw [fill=black] (15.077683537175254,3.) circle (2.5pt);
\draw[color=black] (15.561775529998547,3.674521525635942) node {$6$};
\draw [fill=black] (13.902113032590307,4.618033988749895) circle (2.5pt);
\draw[color=black] (14.20030080326626,5.444438670387918) node {$5$};
\draw [fill=black] (10.097886967409693,6.618033988749895) circle (2.5pt);
\draw[color=black] (10.433554059306928,7.441268269595276) node {$2$};
\draw [fill=black] (8.922316462824746,5.) circle (2.5pt);
\draw[color=black] (8.436724460099573,5.762116106625452) node {$3$};
\draw [fill=black] (10.097886967409693,3.381966011250105) circle (2.5pt);
\draw[color=black] (9.43513925970325,3.4929915620716363) node {$4$};
\draw [fill=black] (10.,2.) circle (2.5pt);
\draw[color=black] (9.026696841683563,2.2222818171214995) node {$10$};
\draw [fill=black] (12.,6.) circle (2.5pt);
\draw[color=black] (12.248853694949979,6.8512958880112835) node {$1$};
\end{scriptsize}
\end{tikzpicture} 
\begin{tikzpicture}[scale=.5][line cap=round,line join=round,>=triangle 45,x=1.0cm,y=1.0cm]
\clip(3.5,-8.) rectangle (15.,-1.);
\fill[color=white] (10.,-4.) -- (10.,-6.) -- (11.902113032590307,-6.618033988749895) -- (13.077683537175254,-5.) -- (11.902113032590307,-3.381966011250105) -- cycle;
\fill[color=white](4.854972311947651,-2.95708556678168) -- (4.854972311947651,-4.95708556678168) -- (6.757085344537957,-5.575119555531575) -- (7.932655849122904,-3.957085566781681) -- (6.757085344537959,-2.3390515780317855) -- cycle;
\fill[color=white] (12.,4.) -- (12.,2.) -- (13.902113032590307,1.381966011250105) -- (15.077683537175254,3.) -- (13.902113032590307,4.618033988749895) -- cycle;
\draw (10.,-4.)-- (10.,-6.);
\draw (10.,-6.)-- (11.902113032590307,-6.618033988749895);
\draw (11.902113032590307,-6.618033988749895)-- (13.077683537175254,-5.);
\draw (13.077683537175254,-5.)-- (11.902113032590307,-3.381966011250105);
\draw (11.902113032590307,-3.381966011250105)-- (10.,-4.);
\draw (4.854972311947651,-2.95708556678168)-- (4.854972311947651,-4.95708556678168);
\draw (4.854972311947651,-4.95708556678168)-- (6.757085344537957,-5.575119555531575);
\draw (6.757085344537957,-5.575119555531575)-- (7.932655849122904,-3.957085566781681);
\draw (7.932655849122904,-3.957085566781681)-- (6.757085344537959,-2.3390515780317855);
\draw (6.757085344537959,-2.3390515780317855)-- (4.854972311947651,-2.95708556678168);
\draw (12.,4.)-- (12.,2.);
\draw (12.,2.)-- (13.902113032590307,1.381966011250105);
\draw (13.902113032590307,1.381966011250105)-- (15.077683537175254,3.);
\draw (15.077683537175254,3.)-- (13.902113032590307,4.618033988749895);
\draw (13.902113032590307,4.618033988749895)-- (12.,4.);
\draw (10.,2.)-- (12.,4.);
\draw (7.932655849122904,-3.957085566781681)-- (10.,-4.);
\draw (13.902113032590307,4.618033988749895)-- (12.,2.);
\draw (15.077683537175254,3.)-- (12.,4.);
\draw (13.902113032590307,4.618033988749895)-- (13.902113032590307,1.381966011250105);
\draw (15.077683537175254,3.)-- (12.,2.);
\draw (13.902113032590307,1.381966011250105)-- (12.,4.);
\draw (10.,-4.)-- (13.077683537175254,-5.);
\draw (10.,-4.)-- (11.902113032590307,-6.618033988749895);
\draw (11.902113032590307,-3.381966011250105)-- (10.,-6.);
\draw (11.902113032590307,-3.381966011250105)-- (11.902113032590307,-6.618033988749895);
\draw (10.,-6.)-- (13.077683537175254,-5.);
\draw (10.097886967409693,6.618033988749895)-- (8.922316462824746,5.);
\draw (8.922316462824746,5.)-- (10.097886967409693,3.381966011250105);
\draw (10.097886967409693,3.381966011250105)-- (12.,4.);
\draw (12.,6.)-- (10.097886967409693,6.618033988749895);
\draw (12.,4.)-- (12.,6.);
\begin{scriptsize}
\draw [fill=black] (10.,-4.) circle (2.5pt);
\draw[color=black] (9.79819918683186,-3.087469617134429) node {10};
\draw [fill=black] (10.,-6.) circle (2.5pt);
\draw[color=black] (9.29899178703002,-5.9011840523811605) node {8};
\draw [fill=black] (11.902113032590307,-6.618033988749895) circle (2.5pt);
\draw[color=black] (12.47576614940536,-6.491156433965153) node {7};
\draw [fill=black] (13.077683537175254,-5.) circle (2.5pt);
\draw[color=black] (13.655710912573344,-4.902769252777482) node {6};
\draw [fill=black] (11.902113032590307,-3.381966011250105) circle (2.5pt);
\draw[color=black] (12.47576614940536,-2.8151746717879713) node {5};
\draw [fill=black] (4.854972311947651,-2.95708556678168) circle (2.5pt);
\draw[color=black] (4.170770316338401,-2.8605571626790476) node {2};
\draw [fill=black] (4.854972311947651,-4.95708556678168) circle (2.5pt);
\draw[color=black] (4.080005334556248,-4.948151743668558) node {3};
\draw[color=black] (7.256779696931588,-1.7259948904021396) node {1};
\draw [fill=black] (6.757085344537957,-5.575119555531575) circle (2.5pt);
\draw[color=black] (6.076834933763605,-5.810419070599008) node {4};
\draw [fill=black] (7.932655849122904,-3.957085566781681) circle (2.5pt);
\draw[color=black] (8.255194496535267,-3.132852108025505) node {9};
\draw [fill=black] (6.757085344537959,-2.3390515780317855) circle (2.5pt);
\draw [fill=black] (12.,4.) circle (2.5pt);
\draw[color=black] (12.339618676732131,4.854466288803926) node {$9$};
\draw [fill=black] (12.,2.) circle (2.5pt);
\draw[color=black] (11.431968858910606,1.7684569082107366) node {$8$};
\draw [fill=black] (13.902113032590307,1.381966011250105) circle (2.5pt);
\draw[color=black] (14.56336073039487,1.405396981082126) node {$7$};
\draw [fill=black] (15.077683537175254,3.) circle (2.5pt);
\draw[color=black] (15.561775529998547,3.674521525635942) node {$6$};
\draw [fill=black] (13.902113032590307,4.618033988749895) circle (2.5pt);
\draw[color=black] (14.20030080326626,5.444438670387918) node {$5$};
\draw [fill=black] (10.097886967409693,6.618033988749895) circle (2.5pt);
\draw[color=black] (10.433554059306928,7.441268269595276) node {$2$};
\draw [fill=black] (8.922316462824746,5.) circle (2.5pt);
\draw[color=black] (8.436724460099573,5.762116106625452) node {$3$};
\draw [fill=black] (10.097886967409693,3.381966011250105) circle (2.5pt);
\draw[color=black] (9.43513925970325,3.4929915620716363) node {$4$};
\draw [fill=black] (10.,2.) circle (2.5pt);
\draw[color=black] (9.026696841683563,2.2222818171214995) node {$10$};
\draw [fill=black] (12.,6.) circle (2.5pt);
\draw[color=black] (12.248853694949979,6.8512958880112835) node {$1$};
\end{scriptsize}
\end{tikzpicture}
\end{center}
In the graph on the left, let $A=\{1,2,3,4,9,10\}$, let $x=9$ and let $y=10$.
It is straightforward to verify that with this assignment, the graphs above are related via Proposition \ref{algorithm}, and thus their respective Laplacian simplices are lattice equivalent.
\end{example}

\begin{remark}
It is not obvious which graph operations, aside from the transformations detailed in the proof of Proposition~\ref{algorithm} and those found in Proposition~\ref{tail}, will result in unimodularly equivalent Laplacian simplices. 
It would be interesting to investigate this phenomenon further.
\end{remark}

We next provide in Theorem~\ref{thm:bridge} an operation on graphs that preserves reflexivity of Laplacian simplices.
We will require the following lemma.

\begin{lemma}\label{keylemma} 
Let $A \in \Z^{k \times k}$.
If $\left(\det{A}\right)$ divides $mC_{k i}$ for each $i \in [k]$, where  $C_{k i}$ is the cofactor of $A$, and $Ax = \1$ has an integer solution $x \in \Z^{k}$, then $Aw = [1, \ldots, 1, 1+m]^T$ has an integer solution $w \in \Z^k$.
\end{lemma}

\begin{proof}
Notice we can write
\[
Aw = A(x+y) = Ax + Ay= \begin{bmatrix} 
1 \\
\vdots \\
1 \\
1+m
\end{bmatrix}
= \begin{bmatrix} 
1 \\
\vdots \\
1 \\
1 \\
\end{bmatrix} + \begin{bmatrix} 
0 \\
\vdots \\
0 \\
m
\end{bmatrix}.
\]
Solving the system $Ay = [0, \ldots, 0, m]^T$ yields
\[
y = A^{-1} \cdot \begin{bmatrix}
0 \\
\vdots \\
0 \\
m \\
\end{bmatrix} = \frac{1}{\det{A}} C^T \cdot \begin{bmatrix}
0 \\
\vdots \\
0 \\
m \\
\end{bmatrix} = \frac{m}{\det{A}} \begin{bmatrix}
C_{k 1} \\
C_{k 2} \\
\vdots \\
C_{k k} \\
\end{bmatrix} 
\]
in which $C_{k i}$ is the cofactor of $A$. 
The above is an integer for each $i \in [k]$ by assumption.
Set $w_j = x_j + y_j \in \Z$, and the result follows.
\end{proof}

We apply Lemma \ref{keylemma} when considering a connected graph $G$ on $m=n$ vertices with $A = L_B(i \mid \emptyset)$ for any $i \in [n]$.
Here $\det{L_B(i \mid \emptyset)}= \pm \kappa$.
Observe in this case the condition $Ax = \1$ for all $i \in [n]$ is equivalent to $T_G$ being a reflexive Laplacian simplex.

\begin{theorem}\label{thm:bridge}
Let $G$ and $G'$ be graphs with vertex set $[n]$ such that $T_G$ and $T_{G'}$ are reflexive.
Suppose $\kappa_G$ divides $nM_{ij}$ and $\kappa_{G'}$ divides $nM_{ij}'$ for all $i,j \in [n-1]$, where $M_{ij} = \det{L_B(i,n \mid j)}$ with $L$ as the Laplacian matrix of $G$, and $M_{ij}'$ is defined similarly. 
Let $H$ be the graph formed by $G$ and $G'$ with $V(H)= V(G) \uplus V(G')$ and $E(H)= E(G) \uplus E(G') \uplus \{i,i'\}$ where $i \in V(G)$ and $i' \in V(G')$. 
Then $T_{H}$ is reflexive.
\end{theorem}

\begin{proof}
To show $T_H$ is reflexive, we show $T_H^*$ is a lattice simplex.
Label the vertices of $H$ such that $V(G) = [n]$, $V(G')=[2n] \setminus [n]$. 
Let $L_B$, $L_B(G)$, and $L_B(G')$ be the Laplacian matrices with basis $B$ of the graphs $H$, $G$ and $G'$, respectively. 
Then $L_B$ is of the form
\[
\left[
\begin{array}{ccc|r|ccc}
     &  &  & 0 & & & \\
    & L_B(G)  &  & \vdots &  & 0 &   \\
     & &  &  0 &  & & \\
     & & &  1 &  & &\\ \cline{1-3} \cline{5-7}
     & 	& &	-1 &  & &\\  
      & 0 &  & 0 &  & L_B(G')  &\\
      &  & & \vdots  &  & & \\
     & &  &  0 & & & \\
\end{array}
\right].
\]
For $1 \le i \le 2n$, the vertex $v_i$ of $T_H^*$ is the solution to $L_B(i \mid \emptyset)v_i = \1$.
We consider two cases: $i \in [n-1]$ and $i = n$.
The cases $i =n+1$ and $i \in [2n] \setminus [n+1]$ follow without loss of generality.

First suppose $i \in [n-1]$. 
Then $L_B(i \mid \emptyset) v_i = \1$ can be solved the following way.
Multiply each side of the equation on the left by the $(2n-1) \times (2n-1)$ unimodular matrix
\[
\left[
\begin{array}{ccc|cc|cccc}
     &  &  & 0 & 0 & & & &\\
    & I_{n-2} & & \vdots & \vdots &  & 0 & &   \\
     & &  &  0 & 0 & & & &\\ \cline{1-3}
     & & &  1 & 1 & 1 & \cdots &\cdots & 1 \\  
     & 	& &	0 & 1 & 1 & \cdots & \cdots & 1 \\  \cline{6-9}
      & 0 &  & 0 & 0 &  & & &\\
      &  & & \vdots & \vdots &  & I_{n-1} & &\\
     & &  &  0 & 0 & & & &\\
\end{array}\right] 
\]
to obtain
\begin{equation*}
\begin{split}
\left[
\begin{array}{ccc|c|ccc}
     &  &  & 0 & & & \\
    & L_B(G)(i \mid \emptyset)  &  & \vdots &  & 0 &   \\
     & &  &  \vdots &  & & \\
     & & &  0 &  & &\\ \cline{1-3} 
     & 	& &	-1 & 0 & \cdots & 0 \\  \cline{5-7}
      &  &  & 0 &  &  &\\
      &  0 & & \vdots  &  & L_B(G')(1 \mid \emptyset) & \\
     &  & & \vdots  &  & & \\
     & &  &  0 & & & \\
\end{array}\right] v_i &= \begin{bmatrix}
1 \\
\vdots \\
1 \\
n+1 \\
n \\
\hline
1 \\
\vdots \\
\vdots \\
1 \\
\end{bmatrix}.
\end{split}
\end{equation*}
We write $(v_i)_k$ to denote the $k\ss{th}$ coordinate of $v_i$.
Then $(v_i)_k \in \Z$ for all $k \in [n-1]$ by Lemma \ref{keylemma}.
Observe from the above multiplication $(v_i)_n=-n$.
Finally, $(v_i)_k \in \Z$ for all $k$, $n+1 \le k \le 2n-1$, as a consequence of $T_{G'}$ being reflexive, i.e., $T_{G'}^*$ is a lattice polytope.

Now suppose $i = n$.
Replace $I_{n-2}$ with $I_{n-1}$ and follow the same argument as above.
Then $(v_i)_n=-n$, and it follows all other coordinates of $v_i$ are integers since $T_G^*$ and $T_{G'}^*$ are lattice polytopes.
\end{proof}

\begin{remark}
The bridge graph construction described in Theorem~\ref{thm:bridge} can be obtained by applying Proposition~\ref{algorithm} to the wedge of two graphs $G$ and $G'$ with a leaf attached to the wedge point. 
Proposition \ref{algorithm} shows the Laplacian simplex associated to the wedge of $G$ and $G'$ is lattice equivalent to the Laplacian simplex associated to the bridge of $G$ and $G'$.
Thus, the wedge of $G$ and $G'$ is reflexive if $G$ and $G'$ satisfy the conditions in Theorem~\ref{thm:bridge}.
\end{remark}

The following proposition shows that Theorem \ref{thm:bridge} applies to graphs such as the one given in Example~\ref{5}.

\begin{prop}
If $G=C_{2k+1}$ and $G'=K_{2k+1}$, then the bridge graph between these is reflexive.
\end{prop}

\begin{proof}
For cyclic graphs on $n$ vertices, the number of spanning trees is $n$.
This and Lemma~\ref{lem:completebridge} show that both cyclic graphs and complete graphs satisfy the condition $\kappa$ divides $|V(G)|\cdot M_{ij}$, as described in Lemma \ref{keylemma}.
We show in later sections that $T_{K_n}$ and $T_{C_{2k+1}}$ are reflexive Laplacian simplices.
\end{proof}

\begin{lemma}\label{lem:completebridge}
For all $n \ge 1$, $G=K_n$ satisfies the conditions of Lemma \ref{keylemma}; that is, for each $i \in [n-1]$, $\kappa$ divides $nM_{nj}$ for each $1 \le j \le n-1$.
Here $M_{nj} = \det{L_B(i, n \mid j)}$. 
\end{lemma}

\begin{proof}
It is sufficient to show for each $1 \le i,j \le n-1$, $\kappa$ divides $nM_{ij}$ where $M_{ij} = \det{L(i, n \mid j, n)}$.
By Lemma \ref{determinant}, this implies the result.
For $G=K_n$, recall Cayley's formula yields $\kappa = n^{n-2}$.
Then we must show $n^{n-3}$ divides $M_{ij}$
There are two cases to consider.

Suppose $i=j$.
Then using row operations on $L(i,n \mid i,n) \in \Z^{(n-2) \times (n-2)}$ which preserve the determinant, we have 

\begin{equation*}
\begin{split}
M_{ii} &=
\det{\left[
\begin{array}{rrrrr}
(n-1) & -1 & \cdots & \cdots & -1 \\
-1 & \ddots & \ddots & & \vdots \\
\vdots & \ddots & \ddots &\ddots & \vdots \\
\vdots &  & \ddots &\ddots  & -1 \\
-1 & \cdots & \cdots & -1 & (n-1) \\
\end{array}
\right]} \\
&= \det{ \begin{bmatrix}
2 & 2 & \cdots & \cdots & 2 \\
-1 & (n-1) & -1 & \cdots & -1 \\
\vdots & \ddots & \ddots & \ddots & \vdots \\
\vdots &  & \ddots & \ddots & -1 \\
-1 & \cdots & \cdots & -1 &(n-1) \\
\end{bmatrix}} \\
&= 2\det{ \begin{bmatrix}
1 & 1 & \cdots & \cdots & 1 \\
-1 & (n-1) & -1 & \cdots & -1 \\
\vdots & \ddots & \ddots & \ddots & \vdots \\
\vdots &  & \ddots & \ddots & -1 \\
-1 & \cdots & \cdots & -1 & (n-1) \\
\end{bmatrix}} \\
&= 2\det{ \begin{bmatrix}
1 & 1 & \cdots & \cdots & 1 \\
0 & n & 0 & \cdots & 0 \\
\vdots & 0 & n & \ddots & \vdots \\
\vdots &  & \ddots & \ddots & 0 \\
0 & \cdots & \cdots & 0 & n \\
\end{bmatrix}} \\
&= 2n^{n-3}.
\end{split}
\end{equation*}

In the case $i \ne j$, $L(i,n \mid j, n) \in \Z^{(n-2) \times (n-2)}$ contains exactly one row and one column with all entries of $-1$. 
Without loss of generality we have 

\begin{equation*}
\begin{split}
M_{ij} &= \det{\begin{bmatrix}
-1 & -1 & \cdots & \cdots & -1 \\
-1 & (n-1) & -1 & \cdots & \vdots \\
\vdots & -1 & \ddots & \ddots & \vdots \\
\vdots &  & \ddots & \ddots & -1 \\
-1 & \cdots & \cdots & -1 & (n-1) \\
\end{bmatrix}} \\
&= \det{\begin{bmatrix}
-1 & -1 & \cdots & \cdots & -1 \\
0 & n & 0 & \cdots & 0 \\
\vdots & 0 & n & \ddots & \vdots \\
\vdots & & \ddots & \ddots & 0 \\
0 & \cdots & \cdots & 0 & n \\
\end{bmatrix}} \\ 
&= -n^{n-3}.
\end{split}
\end{equation*}

\end{proof}

\section{Trees}

Consider the case where $G=P_k$, a path on $k$ vertices.
Label the vertices along the path with the elements of $[k]$ in increasing order.
Then $L$ and consequently $L_B$ have the form
\[
L = 
\left[
\begin{array}{rrrrrr}
    1 & -1 & 0 & \cdots & \cdots & 0 \\
    -1 & 2 & -1& 0 &  & \vdots  \\
    0 & -1 & 2 & -1 & \ddots & \vdots \\
    \vdots & \ddots & \ddots & \ddots & \ddots & 0 \\
    \vdots &  & \ddots & \ddots & 2 & -1 \\
    0 & \cdots & \cdots & 0 & -1 & 1
\end{array}
\right] \hspace{.2in}
L_B = 
\left[
\begin{array}{rrrrrr}
    1 & 0 & \cdots & \cdots  & 0 \\
    -1 & 1 & 0  &  & \vdots  \\
    0 & -1 & 1 & \ddots & \vdots \\
    \vdots & \ddots & \ddots & \ddots &  0 \\
    \vdots &  & \ddots & -1 & 1 \\
    0 & \cdots & \cdots & 0 & -1
\end{array}
\right]
\]

Observe that multiplication by the lower triangular matrix of all ones yields

\[
L_B \cdot \begin{bmatrix}
    1 & 0 & \cdots & \cdots & 0 \\
    \vdots & \ddots & \ddots &  & \vdots  \\
    \vdots &  & \ddots & \ddots & \vdots \\
    \vdots  &  &  & \ddots & 0 \\
    1 & \cdots & \cdots & 1 & 1
\end{bmatrix}  =
\begin{bmatrix}
    1 & 0 & \cdots & \cdots & 0 \\
    0 & 1 & \ddots &  & 0 \\
    \vdots & \ddots & \ddots & \ddots & \vdots \\
    \vdots &  & \ddots & \ddots & 0 \\
    0 & \cdots & \cdots & 0 & 1 \\
    -1 & \cdots & \cdots & -1 & -1
\end{bmatrix} \]

Since the lower triangular matrix is an element in ${\text GL}_{k-1}(\Z)$, it follows that $T_P$ is lattice equivalent to 
\[
S_{k-1}(1) := \conv{ e_1, e_2, \cdots , e_{k-1}, - \sum_{i=1}^{k-1} e_i} \, .
\] 
We leave it as an exercise for the reader to show that $S_{k-1}(1)$ is the unique reflexive $(k-1)$-polytope of minimal volume.
This extends to all trees as follows.

\begin{prop}\label{prop:trees}
Let $G$ be a tree on $n$ vertices. 
Then $T_G$ is unimodularly equivalent to $S_{n-1}(1)$.
\end{prop}

\begin{proof} 
Let $G$ be a tree on $n$ vertices. 
Then $T_G$ is a simplex that contains the origin in its strict interior and has normalized volume equal to $n$, since $G$ has only one spanning tree.
Consider the triangulation of $T_G$ that consists of creating a pyramid at the origin over each facet. 
Since $G$ is a tree, 
\[
\text{vol}(T_G) = \sum_{\text{facet}} \text{vol}(F) = 1\cdot n = n \, .
\]
There are $n$ facets, so each must have vol$(F)=1$. 
Applying a unimodular transformation to $n-1$ of the vertices of $T_G$, we can assume that the vertices of $T_G$ are the $n$ standard basis vectors and a single integer vector in the strictly negative orthant (so that the origin is in the interior of $T_G$).
Because the normalized volume of the pyramid over each facet is equal to $1$, it follows that the final vertex is $-\1$.
\end{proof}

\begin{corollary}\label{cor:treeunim}
The $h^*$-vector of the Laplacian simplex for any tree is $(1,1,\ldots,1)$, hence is unimodal.
\end{corollary}

\begin{corollary}\label{tree transform} 
Let $G$ be a tree on $n$ vertices with Laplacian matrix $L_B$.
Then there exists $U \in {\text GL}_{n-1}(\Z)$ such that 
\[
L_B \cdot U = 
\left[
\begin{array}{rrrr}
    1 & 0 & \cdots &  0 \\
    0 & 1 & \ddots & \vdots  \\
    \vdots & \ddots & \ddots & 0 \\
    0 & \cdots & 0 & 1 \\
    -1 & \cdots & \cdots & -1 
\end{array} 
\right]
\] 
\end{corollary}

The next proposition asserts that attaching an arbitrary tree with $k$ vertices to a graph on $n$ vertices yields a lattice isomorphism between the resulting Laplacian simplex and the Laplacian simplex obtained by attaching any other tree with $k$ vertices at the same root.

\begin{prop}\label{tail} 
Let $G$ be a connected graph on $n$ vertices, and let $v$ be a vertex of $G$. 
Let $G'$ be the graph obtained from $G$ by attaching $k$ vertices such that $G'$ restricted to the vertex set $ \{ v \} \cup [k]$ forms a tree, call it $T$. 
The edges of $G'$ are the edges from $G$ along with any edges among the vertices $\{v\} \cup [k]$.
Let $P$ be the graph obtained from $G$ by attaching $k$ vertices such that $P$ restricted to the vertex set $\{ v\} \cup [k]$ forms a path. 
Then $T_{G'} \cong T_{P}.$ 
\end{prop}

\begin{proof} 
The reduced Laplacian matrix associated to $T_{G'}$ is the following $(n+k) \times (n+k-1)$ matrix:
\[ 
\left[
\begin{array}{ccc|ccc}
     &  &  &  &   & \\
     & L_B(G) & &  & 0 &   \\
     & &  &  &  &   \\
     & & &   &  &   \\ \cline{4-6}
     & 	& &	 & 	&   \\ \cline{1-3}
      &  &   & &  &   \\
      & 0 &  &  & L_B(T) &  \\
      &  &  & & &  \\
\end{array}
\right] 
\]

Here $L_B(T) \in \Z^{(k+1) \times k}$ is the Laplacian matrix for $T$, the tree on $(k+1)$ vertices.
Let $U \in {\text GL}_{k}(\Z)$ be the matrix such that $L_B(T) \cdot U$ gives the matrix with vertex set $S_{k}(1)$ as in Corollary~\ref{tree transform}. Then we have 

\[ 
\left[
\begin{array}{ccc|ccc}
     &  &  &  &  &  \\
     & L_B(G) & &  & 0 &   \\
     &  &  &  &  &   \\
     &  &  &  &  &  \\ \cline{4-6}
     & 	&  &  &  &  \\ \cline{1-3}
     &  &  &  &  & \\
     & 0 &  &  & L_B(T) & \\
     &  &  &  & &  \\
\end{array}
\right] \cdot \left[ \begin{array}{ccc|ccc}
	 &  &  &  &  &  \\
     & I_{n-1} &  &  & 0 &   \\
     & &  &  &  &  \\ \cline{1-6}
     & 	&	&  & 	&  \\ 
     &  &  & &  & \\
     & 0 &  & & U & \\
     &  & &  & &  \\
\end{array}
\right] = \left[ \begin{array}{ccc|ccc}
     &  &  &  &  &  \\
     & L_B(G) & &  & 0 &   \\
     &  &  &  &  &   \\
     &  &  &  &  &  \\ \cline{4-6}
     & 	&  &  &  &  \\ \cline{1-3}
     &  &  &  &  & \\
     & 0 &  &  & L_B(P) & \\
     &  &  &  & &  \\
\end{array}
\right]
\]

For any set of $k$ vertices we attach to a vertex $v \in V(G)$ to obtain a tree on the vertex set $\{v\} \cup [k]$, we get a corresponding unimodular matrix $U$ such that the above multiplication holds. 
The determinant of the $(n-1+k) \times (n-1+k)$ transformation matrix is equal to the determinant of $U$, which is $\pm 1$. 
Then $T_{G'}$ is lattice equivalent to $T_P$ for any such $G'$.
\end{proof}

\begin{remark}
It follows from Theorem~\ref{thm:bridge} that bridging a tree to a graph $G$ with $T_G$ reflexive and $L(G)$ satisfying the appropriate division condition on minors will result in a new reflexive Laplacian simplex.
Further, Proposition~\ref{tail} shows that the equivalence class of the resulting reflexive simplex is independent of the choice of tree used in the attachment.
\end{remark}

\section{Cycles}

Let $C_n$ denote the cycle with $n$ vertices.
In this section, we show that odd cycles are reflexive and have unimodal $h^*$-vectors, but fail to be IDP.
We show that whiskering even cycles results in reflexive Laplacian simplices.
Finally, we determine the $h^*$-vectors for $T_{C_n}$ when $n$ is an odd prime.

\subsection{Reflexivity and Whiskering}

\begin{theorem}\label{cycle} 
For $n \ge 3$, the simplex $T_{C_n}$ is reflexive if and only if $n$ is odd.
For $k \ge 2$, the simplex $T_{C_{2k}}$ is $2$-reflexive. 
\end{theorem}

\begin{proof} 
Let $C_n$ be a cycle with vertex set $[n]$ and vertices labeled cyclically.
Then $L$ and consequently $L_B$ have the form (when rows and columns are suitably labeled)
\[
L = \left[ \begin{array}{rrrrrr}
    2 & -1& 0  & \cdots & 0 & -1 \\
    -1& 2 & -1 & \ddots &  & 0 \\
    0 & -1& 2 & -1  & \ddots & \vdots \\
    \vdots  & \ddots & \ddots  & \ddots & \ddots & 0 \\
    0 & & \ddots &-1 & 2 & -1 \\
    -1 & 0 & \cdots & 0 & -1 & 2
\end{array}\right] \hspace{.2in}
L_B = \left[ \begin{array}{rrrrr}
    2 & 1 &  \cdots & \cdots & 1 \\
    -1& 1 &  0 & \cdots & 0 \\
    0 & -1&   1  & \ddots & \vdots \\
    \vdots  & \ddots &  \ddots & \ddots & 0 \\
    0 & \cdots & 0  & -1 & 1 \\
    -1 & -1 & \cdots & -1 & -2
\end{array}\right] \, . 
\]

To show that $T_{C_n}$ is reflexive, we show $T_{C_n}^*= \{x \mid L_B x \le \1 \}$ is a lattice polytope.
Each intersection of $(n-1)$ of these facet hyperplanes will yield a unique vertex of $T_{C_n}^*$, since the rank of $L_B$ is $n-1$.
For each $i \in [n]$, let $v_i \in \R^{n-1}$ be the vertex that satisfies $L_B(i \mid \emptyset) \cdot v_i = \1$.
Solving the appropriate system of linear equations yields
\begin{equation*}
\begin{split}
v_1 &= \left( \frac{1-n}{2}, \frac{3-n}{2}, \frac{5-n}{2}, \cdots , \frac{n-5}{2}, \frac{n-3}{2} \right) = \left( \frac{(2j-1)-n}{2} \right)_{j=1}^{n-1} \\
v_i &= \left( \left( \frac{(2j+1)+n-2i}{2} \right)_{j=1}^{i-1}, \left(\frac{(2j+1) -n - 2i}{2} \right)_{j=i}^{n-1} \right), \text{ for } 2 \le i \le n-1 \\
v_n &= \left( \frac{3-n}{2} , \frac{5-n}{2}, \frac{7-n}{2}, \cdots , \frac{n-3}{2}, \frac{n-1}{2} \right)= \left(\frac{(2j+1)-n}{2} \right)_{j=1}^{n-1} \\
\end{split}
\end{equation*}

These are the vertices of $T_{C_n}^*$.
Note $v_i \in \Z^{n-1}$ only if $n$ is odd. 
Then $T_{C_n}$ is reflexive if and only if $n$ is odd.

For the even case, observe the coordinates of each vertex of $2\cdot T_{C_{2k}}^*$ are relatively prime.
Then each of these vertices is primitive.
Thus, for $n=2k$ each vertex of $T_{C_{2k}}^*$ is a multiple of $\frac{1}{2}$, which allows us to write
\[
T_{C_{2k}} = \left\{ x \mid \frac{1}{2} \tilde{A}x \le \1 \right\}= \left\{x \mid \tilde{A}x \le 2 \cdot \1 \right\} 
\] 
where $\tilde{A} \in \Z^{n \times (n-1)}$.
The facets of $T_{C_{2k}}$ have supporting hyperplanes $\langle r_i, x \rangle = 2$ where $r_i$ is the $i$\ss{th} row of $\tilde{A}$.
Thus $T_{C_{2k}}$ is a  $2$-reflexive Laplacian simplex.
\end{proof}

\begin{example} 
Below are the dual polytopes to $T_{C_n}$ for small $n$.
\begin{itemize}
\item $T_{C_3}^* = \conv{(-1,0), (1,-1), (0,1)}$
\item $T_{C_4}^* = \conv{ (-\frac{3}{2}, - \frac{1}{2}, \frac{1}{2}), (\frac{3}{2}, -\frac{3}{2}, -\frac{1}{2}), (\frac{1}{2}, \frac{3}{2}, - \frac{3}{2}), (-\frac{1}{2}, \frac{1}{2}, \frac{3}{2})}$
\item $T_{C_5}^* = \conv{(-2,-1,0,1), (2,-2,-1,0), (1,2,-2,-1) , (0,1,2,-2), (-1,0,1,2)}$
\end{itemize}
\end{example}


Although $T_{C_{2k}}$ is not reflexive, we show next that whiskering $C_{2k}$ results in a graph $W(C_{2k})$ such that $T_{W(C_{2k})}$ is reflexive. 
The technique of whiskering graphs has been studied previously in the context of Cohen-Macaulay edge ideals, see \cite[Theorem 4.4]{DochtermannEngstromEdgeIdeals} and \cite{VillarrealCMGraphs}.

\begin{definition}\label{whisker}
To add a \emph{whisker} at a vertex $x \in V(G)$, one adds a new vertex $y$ and the edge connecting $x$ and $y$. 
Let $W(G)$ denote the graph obtained by whiskering all vertices in $G$.
We call $W(G)$ the \emph{whiskered graph of $G$}. 
If $V(G) = \{ x_1, \ldots, x_n\}$ and $E(G) = E$, then $V(W(G)) = V(G) \cup \{y_1, \ldots, y_n\}$ and $E(W(G))=E \cup \{ \{x_1, y_2\}, \ldots, \{x_n, y_n \} \}$.
\end{definition}

\begin{prop}\label{evenreflexive}
$T_{W(C_{n})}$ is reflexive for even integers $n \ge 2$.
\end{prop}

\begin{proof}
$W(C_n)$ is a graph with vertex set $[2n]$ and $2n$ edges.
Label the vertices of the cycle with $[n]$ in a cyclic manner.
Label the vertices of each whisker with $i$ and $n+i$ where $i\in [n]$. 
The Laplacian matrix has the following form.
\[
L=\left[\begin{array}{ ccc | ccc}
 &  &  & & &   \\
  &  L + I_n&  & &-I_n &   \\
   &  &  & & &   \\
   \hline
    &  &  & & &   \\ 
     &  -I_n&  & & I_n&   \\
      &  &  & & &   \\
\end{array}\right]
\]
Consequently if $A$ is the $n \times (n-1)$ matrix given by Equation~\eqref{eqn:A}, then   
\[
L_B=\left[\begin{array}{ ccc | cccc}
 &  &  & & &  &  \\
  &  L_B(C_n) + A&  & & A^T & &  \\
     &  &  & & &  & \\
   &  &  & 1 & \cdots & \cdots & 1\\
   \hline
    &  &  & & &   & \\ 
     &  -A &  & & -A^T& &  \\
      &  &  & & &  & \\
      &  &  & -1 &\cdots & \cdots &-1 \\
\end{array}\right].
\]
We show $T_{W(C_n)}$ is reflexive by showing $T_{W(C_n)}^*$ is a lattice polytope.
Each vertex of the dual is a solution to $L_B(i \mid \emptyset) v_i = \1$.
We consider the following cases.

\textbf{Case:} $1 \le i \le n$.
Multiply both sides of $L_B(i \mid \emptyset) v_i = \1$ by the $(2n-1) \times (2n-1)$ upper diagonal matrix with the following entries. 

\[
x_{\ell k}=
\begin{cases}
1, & \text{if $\ell = k$} \\
1, & \text{if $\ell < k$ and $\{v_\ell, v_k\}$ is a whisker} \\
-1, & \text{if $n < \ell = k-1$} 
\end{cases}
\]
In this matrix each of the first $n-1$ rows will have exactly two non-zero entries of value $1$, which corresponds to adding the two rows of $L_B(i \mid \emptyset)$ that are indexed by the labels of a whisker in the graph. 
The last $n$ rows will have an entry of $1$ along the diagonal and an entry of $-1$ on the superdiagonal, which corresponds to subtracting consecutive rows in $L_B(i \mid \emptyset)$ to achieve cancellation.
We obtain the following system of linear equations.
\[
\left[\begin{array}{ ccc | cccc}
 &  &  & & &  &  \\
  &  L_B(C_n)(i \mid \emptyset)  & & &  & 0 &  \\
     &  &  & & &  & \\
   &  &  &  &  &  & \\
   \hline
    &  &  & 0 &  &   & \\ 
     &  -I_{n-1} &  & \vdots & & I_{n-1} &  \\
      &  &  & 0 & &  & \\
      0& \cdots &0  & -1 &\cdots & \cdots &-1 \\
\end{array}\right] v_i = \begin{bmatrix}
2 \\
\vdots \\
2 \\
0 \\
\vdots \\
0 \\
1 \\
\end{bmatrix}
\]

Let $(v_i^*)_j$ denote the $j\ss{th}$ coordinate of the vertex $v_i \in \Q^{n-1}$ of $T_{C_n}^*$ described in Proposition \ref{cycle}.
Then the vertex $v_i$ of $T_{W(C_n)}^*$ has the following form. 
\[
(v_i)_j = \begin{cases}
2(v_i^*)_j, & \text{if $1 \le j \le n-1$} \\
-1 - \sum_{k=1}^{n-1} 2(v_i^*)_k, & \text{if $j=n$} \\
2(v_i^*)_{j-n}, & \text{if $n+1 \le j \le 2n-1$} \\
\end{cases}
\]
 Since $2(v_i^*)_j \in \Z$ by Proposition \ref{cycle} for $1 \le j \le n-1$, then $v_i \in \Z^{2n-1}$. \\

\textbf{Case:}  $n+2 \le i \le 2n$.
The strategy is to multiply the equality $L_B( i \mid \emptyset) v_i = \1$ by the matrix that performs the following row operations.
Let $r_m \in \Z^{2n-1}$ denote the $m\ss{th}$ row of $L_B( i \mid \emptyset)$.
For each whisker with vertex labels $\{m, n+m\}$, replace $r_m$ with $r_m + r_{n+m}$ for $m \in [n]$.  
Row $i-n$ will not have a row to add because the index of its whisker is the index of the deleted row.
Since each column in $L_B$ sums to $0$, the negative sum of all the rows of $L_B(i \mid \emptyset)$ is equal to the row removed.
We recover the missing row by replacing $r_{i-n}$ with $-\sum_{k=1}^{2n-1} r_k$ for $r_k \in L_B(i \mid \emptyset)$. 
Then as in the previous case, we want to replace row $r_k$ with $r_k - r_{k+1}$ for $n+1 \le k \le 2n-2$.
Here $r_{i-n}$ plays the role of the deleted $r_i$.
We obtain a similar system of linear equations found in the first case.
The vertex $v_i$ of $T_{W(C_n)}^*$ has the following form.
\[
(v_i)_j = \begin{cases}
2(v_i^*)_j, & \text{if $1 \le j \le n-1$} \\
-1 - \sum_{k=1}^{n-1} 2(v_i^*)_k, & \text{if $j = n$} \\
2(v_i^*)_{j-n}, & \text{if $n+1 \le j \le 2n-1$ and $j \ne i-1, i$} \\
2(v_i^*)_{j-n} + 2n, & \text{if $j=i-1$} \\
2(v_i^*)_{j-n} - 2n, & \text{if $j=i$} \\
\end{cases}
\]

Observe in the case $i=2n$, the last equality is not applicable since $j \in [2n-1]$.
Then $v_i \in \Z^{2n-1}$. \\

\textbf{Case:} $i=n+1$.
Here $(v_i)_{i-1} = (v_i)_{n} = -(2n-1) - \sum_{k=1}^{n-1} 2(v_i^*)_k \in \Z$ and the other coordinates are as described above. 
Then $v_i \in \Z^{2n-1}$.

\end{proof}

We extend Proposition \ref{evenreflexive} to a more general result, that whiskering a graph whose Laplacian simplex is $2$-reflexive results in a graph whose Laplacian simplex is reflexive.
Although even cycles are the only known graph type to result in $2$-reflexive Laplacian simplices, we include the following result.

\begin{prop}
If $G$ is a connected graph on $n$ vertices such that $T_G$ is $2$-reflexive, then $T_{W(G)}$ is reflexive for all $n \ge 2$.
\end{prop} 

\begin{proof}
If $T_G$ is $2$-reflexive, then each vertex $v_i$ of $T_G^*$ satisfies $2 v_i \in \Z^{n-1}$ for each $1 \le i \le n$.
As in the proof of Proposition \ref{evenreflexive}, we can find descriptions of the vertices of $T_{W(G)}^*$ in terms of the coordinates from vertices of $T_G^*$ to show they are lattice points.
The result follows. 
\end{proof}

Given a graph $G$ with $T_G$ is reflexive, we have already seen that attaching a tree on $|V(G)|$ vertices to obtain a new graph $G'$ on $2\cdot |V(G)|$ vertices results in the reflexive Laplacian simplex $T_{G'}$.
Whiskering a graph also preserves the reflexivity of $T_G$, as seen in the following result.

\begin{prop}
If $G$ is a connected graph on $n$ vertices such that $T_G$ is reflexive, then $T_{W(G)}$ is reflexive for all $n \ge 1$.
\end{prop}

\begin{proof}
If $T_G$ is reflexive, then vertices of $T_G^*$ are integer and satisfy $L_B(i \mid \emptyset)v_i = \1$ for all $1 \le i \le n$.
Observe $2v_i \in \Z^{n-1}$ satisfies $L_B(i \mid \emptyset) 2v_i = 2\cdot \mathbbm{1}$.
Following the proof technique in Proposition \ref{evenreflexive}, we can find descriptions of the vertices of $T_{W(G)}^*$ in terms of the coordinates from vertices of $T_G^*$ to show they are lattice points.
\end{proof}

\subsection{$h^*$-Unimodality}

For odd $n$, our proof of the following theorem can be interpreted as establishing the existence of a weak Lefschetz element in the quotient of the semigroup algebra associated to $\cone{T_{C_n}}$ by the system of parameters corresponding to the ray generators of the cone.
This proof approach is not universally applicable, as there are examples of reflexive IDP simplices with unimodal $h^*$-vectors for which this proof method fails~\cite{BraunDavisFreeSum}.

\begin{theorem}\label{unimodal} 
For odd $n$, $h^*(T_{C_n})$ is unimodal. 
\end{theorem}

\begin{proof}
Recall from Lemma~\ref{lem:fpp} that $h^*_i(T_{C_n})$ is the number of lattice points in $\Pi_{T_{C_n}}$ at height $i$. 
Theorem~\ref{cycle} shows $h^*_i(T_{C_n})$ is symmetric.
Our goal is to prove that for $i\leq \lfloor n/2\rfloor$ we have $h_i^*\leq h^*_{i+1}$.
This will show that $h^*(T_{C_n})$ is unimodal.

While $\kappa=n$ for $C_n$, we will freely use both $\kappa$ and $n$ to denote this quantity, as it is often helpful to distinguish between the number of spanning trees and the number of vertices.
Lattice points in the fundamental parallelepiped of $T_{C_n}$ can be described as follows: 
\[
\Z^n \cap \left\{\frac{1}{\kappa n} b \cdot [L_B \mid \1] \mid 0 \le b_i < \kappa n, b_i\in \Z_{\geq 0}, \sum_{i=1}^n b_i \equiv 0 \bmod \kappa n \right\}.
\]
We will use the modular equation above extensively in our analysis.
Denote the height of a lattice point in $\Pi_{T_{C_n}}$ by 
\[
h(b):= \dfrac{\sum_{i=1}^n b_i}{n \kappa} \in \Z_{\ge 0} \, .
\]

We first show that every lattice point in $\Pi_{T_{C_n}}$ arising from $b$ satisfies 
\[ 
\dfrac{(k-j+1)(b_1-b_n)}{\kappa n} + \dfrac{b_j - b_{k+1}}{\kappa n} \in \Z 
\] 
for each $1 \le j < k \le n-1$. 
 Since the lattice point lies in $\Pi_{T_{C_n}}$, we have the following constraint equations:
\[
\frac{b_1 - b_n + b_i - b_{i+1}}{\kappa n} \in \Z
\] 
for each $1 \le i \le n-1$. Summing any consecutive set of these equations where $1 \le j \le k \le n-1$ yields
\[ 
\sum_{i=j}^k \left( \dfrac{b_1-b_n}{\kappa n} + \dfrac{b_i - b_{i+1}}{\kappa n} \right) \in \Z \, . 
\] 
The result follows. 

Thus, each vector $b$ corresponding to an integer point in $\Pi_{T_{C_n}}$ satisfies $\kappa \mid (b_1 - b_n)$, which follows from setting $j=1$ and $k=n-1$.
We next claim that every lattice point in $\Pi_{T_{C_n}}$ arises from $b \in \Z^n$ such that $b_i \equiv b_j \mod(\kappa)$ for each $1 \le i, j \le n$.   
To prove this, set $\frac{b_1-b_n}{\kappa}= B \in \Z.$ 
Then for each $1 \le i \le n-1$, our constraint equation becomes $\frac{B}{n} + \frac{b_i-b_{i+1}}{\kappa n} = C$ for some $C \in \Z$. 
Then $\frac{b_i - b_{i+1}}{\kappa} = Cn - B \in \Z$ holds for each $i$. 
The result follows. 

\textbf{First Major Claim:}
For $n$ odd, any lattice point in $\Pi_{T_{C_n}}$ arises from $b \in \Z^n$ such that $b_i \equiv 0 \mod(\kappa)$ for each $1 \le i \le n$.  

To prove this, let $b_i = m_i \kappa + \alpha$ such that $0 \le m_i < \kappa$ and $0 \le \alpha < \kappa$. 
Constraint equations yield 
\[ 
\frac{b_1-b_n + b_i - b_{i+1}}{\kappa n} = \frac{m_1-m_n+m_i-m_{i+1}}{n} \in \Z 
\] 
using $\kappa = n$. 
Summing all $n-1$ integer expressions with linear coefficients yields 
\[ 
\sum_{i=1}^k i(m_1 - m_n + m_i - m_{i+1}) = \frac{n(n-1)}{2}m_1 +\sum_{i=1}^{n-1} m_i - (n-1)m_n - \frac{n(n-1)}{2}m_n, 
\]
which is divisible by $n$. Call the resulting sum $An$ for some $A \in \Z$.
Finally, notice the last constraint equation (corresponding to $h(b)$) can be written 
\begin{equation*}
\begin{split}
\frac{\sum_{i=1}^n b_i}{\kappa n} &= \frac{\sum_{i=1}^n m_i + \alpha}{n}  \\ 
&= \frac{m_n + An - \frac{n(n-1)}{2}m_1 + (n-1)m_n + \frac{n(n-1)}{2}m_n + \alpha}{n} \in \Z.
\end{split}
\end{equation*}
Then $n$ odd implies $n$ divides $\frac{n(n-1)}{2}$ so that $n$ divides $\alpha$. 
Since $0 \le \alpha < n$, then $\alpha = 0$ as desired.


\textbf{Second Major Claim:}
Consider $T_{C_n}$ for odd $n$. Suppose $h(b) < \frac{n-1}{2}$. 
If $p \in \Pi_{T_{C_n}} \cap \Z^n$, then $ p + (0, \cdots, 0, 1)^T \in \Pi_{T_{C_n}} \cap \Z^n.$ 

To establish this, it suffices to prove that for every $p = \frac{1}{n^2}b \cdot [L_B \mid \1] \in \Pi_{T_{C_n}} \cap \Z^n$ such that $h(b)< \frac{n-1}{2}$, we have $b_i < n(n-1)$ for each $i$.
This would imply 
\[
p + (0, \cdots, 0, 1)^T = \frac{1}{n^2} (b+n \1) \cdot [L_B \mid \1] \in \Pi_{T_{C_n}} \cap \Z^n \, ,
\]
providing an injection from the lattice points in $\Pi_{T_{C_n}}$ at height $i$ to those at height $i+1$.
Constraint equations yield, using the same notation as in the proof of our first major claim, that
\[
-m_{j-1} + 2m_j -  m_{j+1} \in n\Z 
\] 
for each $1 \le j \le n$. 
Note that this comes from subtracting the two integers 
\[ 
\frac{m_1+m_j- m_{j+1} - m_n}{n} - \frac{m_1+m_{j-1} - m_j - m_n}{n} = \frac{2m_j - (m_{j-1} + m_{j+1})}{n} \in \Z
\] 
for each $2 \le j \le n-1$, as well as
\[ 
\frac{2m_1 - m_2 -m_n}{n}, \frac{-(m_1 + m_{n-1} - 2m_n)}{n} \in \Z \, .
\]
For a contradiction, suppose there exists a $j$ such that $b_j = n(n-1)$. 
Then $m_j = n-1$. 
Constraints on the other variables $m_i$ imply 
\[
0 \le \frac{2(n-1)-(m_{j-1}+m_{j+1})}{n} \le 1 \implies 2(n-1)-(m_{j-1} + m_{j+1}) = 0 \text{ or } n.
\]

\noindent Case 1: If the above is $0$, then 
\[ 2(n-1) = m_{j-1} + m_{j+1} \implies m_{j-1} = m_{j+1} = n-1. \] 
Apply these substitutions on other constraint equations to yield $m_i = n-1$ for all $1 \le i \le n$.
Then 
\[ 
h(b) = \frac{\sum_{i=1}^n m_i}{n} = \dfrac{n(n-1)}{n} = n-1 > \dfrac{n-1}{2},
\] 
which is a contradiction.

\noindent Case 2: If the above is $1$, then $n-2=m_{j-1}+m_{j+1}$.
Adding subsequent constraint equations yields 
\begin{equation*}
\begin{split}
\left( -m_j + 2m_{j-1} - m_{j-2}\right) + \left(-m_j + 2m_{j+1} - m_{j+2} \right) &= -2m_j + 2(m_{j-1} + m_{j+1}) - (m_{j-2} + m_{j+2})  \\
&= -2(n-1) + 2(n-2) - (m_{j-2} + m_{j+2}) \\
&= -2 -(m_{j-2} + m_{j+2}) 
\end{split}
\end{equation*}

\noindent Since the above is in $n\Z$, it is equal to either $-2n$ or $-n$.

\noindent Case 2a: If the above is equal to $-2n$, then $m_{j-2}=m_{j+2} = n-1$. 
Then 
\[
-m_{j-3} + 2m_{j-2} - m_{j-1} = -m_{j-3} + m_{j+1} \in n \Z \implies m_{j-3} = m_{j+1}. 
\]
A similar argument shows $m_{j+3} = m_{j-1}$. 
Continuing in this way shows $m_{j \pm k} = m_{j \mp 1}$ for remaining $m_i$. 
Then for each of the $\frac{n-3}{2}$ pairs, $m_{j-k} + m_{j+k}=n-2$ where $k \in \{ 1, \hat{2}, 3, \cdots, \frac{n-1}{2} \}$.
But then 
\begin{equation*}
\begin{split}
 h(b) &= \frac{\sum_{i=1}^n m_i}{n} \\
&= \frac{n-1 + 2(n-1) + \frac{n-3}{2}(n-2)}{n} \\
&= \frac{n+1}{2},
\end{split}
\end{equation*}
which is a contradiction. 

\noindent Case 2b: If the above is equal to $-n$, then $m_{j-2}+m_{j+2}=n-2$.
Adding subsequent constraint equations as above yields $n-2 -(m_{j-3}+m_{j+3})$.
Since the above is in $n\Z$, it is equal to either $-2n$ or $-n$. 

\noindent Case 2b(i): If the above is equal to $-n$, then $m_{j-3}=m_{j+3}=n-1$. 
Following the same argument as Case 2a leads to the contradiction, $h(b) = \dfrac{n+1}{2}$.\\

\noindent Case 2b(ii): If the above is equal to $-2n$, then $m_{j-3} + m_{j+3} = n-2$. 
Continuing in this manner yields $m_{j-k}+m_{j+k} = n-2$ for all $k \in \{ 1, 2, \cdots, \frac{n-1}{2} \}$. 
But then 
\[
h(b) = \frac{n-1 + \frac{(n-1)}{2}(n-2)}{n} = \frac{n-1}{2}, 
\]
which is a contradiction.
This concludes the proof of our second major claim.

The second claim implies that for $i\leq \lfloor n/2 \rfloor$, we have $h_i^*\leq h_{i+1}^*$.
Thus, our proof is complete.
\end{proof}

\subsection{Structure of $h^*$-vectors}

We next classify the lattice points in the fundamental parallelepiped for $T_{C_n}$ by considering the matrix $[L_B \mid \1]$ over the ring $\Z/\kappa \Z$. 
Let 
\[
[\widetilde{L} \mid \1] := [L_B \mid \1] \bmod \kappa \, .
\]
Recall that for a cycle we have $n=\kappa$.

\begin{lemma}\label{kernel}
For $C_{n}$ with odd $n$ and corresponding reduced Laplacian matrix $[L_B \mid \1]$, we have 
\[
\ker_{\Z/\kappa \Z}{[\widetilde{L} \mid \1]} = \{ x \in \left(\Z/\kappa \Z\right)^n \mid x[L_B\mid \1] \equiv \mathbf{0} \mod{\kappa} \} = \langle \1^n, (0, 1, \cdots, n-1) \rangle .
\]


\end{lemma}

\begin{proof}
Consider the second principal minor of $[L_B\mid \1]$ with the first and $n$\ss{th} rows and columns deleted. 
The matrix $[L_B \mid \1](1,n \mid 1,n)$ is the lower diagonal matrix of the following form:
\[ 
\left[
\begin{array}{rrrrr}
    1 & 0&  0 & \cdots & 0  \\
    -1& 1 &  0 &   & \vdots \\
    0 & \ddots  & \ddots & \ddots & \vdots \\
    \vdots  & \ddots & \ddots &  \ddots & 0  \\
    0 & \cdots & 0  & -1 & 1 
\end{array}
\right]
\]
Then $\det{[L_B \mid \1](1,n \mid 1,n)} = 1$ implies there are $n-2$ linearly independent columns, hence $\text{rk}_{\Z/\kappa \Z} [L_B \mid \1] \ge n-2$.

Since the entries in each column of $[L_B\mid \1]$ sum to $0$, then \[\1 \cdot [L_B\mid \1] = (0, 0, \ldots, 0, n) \equiv \bf{0} \mod \kappa\] implies $\1 \in \ker_{\Z/\kappa \Z}{[L_B\mid \1]}$.
Consider \[ (0, 1, \ldots, n-1) \cdot 
\left[\begin{array}{rrrrrrr}
    2 & 1& 1 &  \cdots & \cdots & 1 & 1\\
    -1& 1 & 0&  \cdots & \cdots & 0 & 1 \\
    0 & -1& 1 &  \ddots  &  & \vdots & \vdots\\
    \vdots  & \ddots & \ddots &  \ddots & \ddots& \vdots & \vdots \\
    \vdots  &  & \ddots &  \ddots & \ddots & 0 & \vdots \\
    0 & \cdots & \cdots & 0  & -1 & 1 & 1 \\
    -1 & -1 & -1 & \cdots & -1 & -2 &1
\end{array}\right] 
= \left(-n, \ldots, -n, \frac{n(n-1)}{2}\right) \equiv \bf{0} \mod \kappa. \]
This shows $(0, 1, \ldots , n-1) \in \ker_{\Z/\kappa \Z}{[L_B\mid \1]}$.
Since these two vectors are linearly independent, we have $\text{rk}_{\Z/\kappa \Z} [L_B \mid \1] \leq n-2$.

Thus, the kernel is two-dimensional and we have found a basis.
\end{proof}

\begin{theorem}\label{cyclefpp}
For odd $n \ge 3$, lattice points in $\Pi_{T_{C_n}}$ are of the form 
\[\frac{(\alpha \1 + \beta (0, 1, \ldots, n-1)) \mod{\kappa}}{\kappa} \cdot [L_B \mid \1] \] for all $\alpha, \beta \in \Z / \kappa \Z$.
Thus, $h^*_i(T_G)$ is equal to the cardinality of
\[
\left\{ \frac{(\alpha \1 + \beta (0, 1, \ldots, n-1)) \mod \kappa}{\kappa} \cdot [L_B \mid \1] \mid 0\leq \alpha, \beta<\kappa-1, \frac{1}{\kappa} \sum_{j=0}^{n-1} (\alpha + j \beta \mod \kappa ) = i \right\}.
\]
\end{theorem}

\begin{proof}  
Since $|\Pi_{T_{C_n}} \cap \Z^n | = \sum_{i=0}^{n-1} h^*_i(T_{C_n}) = n \kappa = n^2$, there are $n^2$ lattice points in the fundamental parallelepiped.
Similarly, there are $n^2$ possible linear combinations of $\1$ and $(0,1,2,\ldots,n-1)$ in $\Z/\kappa \Z$.
We show that each such linear combination yields a lattice point.
Recall the sum of the coordinates down each of the first $n-1$ columns of $[L_B \mid \1]$ is $0$. 
Since 
\[
(\alpha \1 + \beta (0, 1, \ldots, n-1)) \cdot [L_B \mid \1] \equiv \mathbf{0} \mod \kappa
\] 
by Lemma \ref{kernel}, it follows that 
\[(\alpha \1 + \beta (0, 1, \ldots, n-1) \mod \kappa) \cdot [L_B \mid \1] \equiv \mathbf{0} \mod \kappa.
\]
Then $\dfrac{(\alpha \1 + \beta (0, 1, \ldots, n-1)) \mod{\kappa}}{\kappa} \cdot [L_B \mid \1]$ is a lattice point.
Since we are reducing the numerators of the entries in the vector of coefficients modulo $\kappa$ prior to dividing by $\kappa$, it follows that each entry in the coefficient vector is greater than or equal to $0$ and strictly less than $1$, and hence the resulting lattice point is an element of $\Pi_{T_{C_n}}$.
\end{proof}

\begin{theorem}\label{primes} 
Consider $C_n$ where $n \ge 3$ is odd. 
Let $n = p_1^{a_1} p_2^{a_2} \cdots p_k^{a_k}$ be the prime factorization of $n$ where $p_1 > p_2 > \cdots > p_k$. Then
\[
h^*(T_{C_n}) = (1, \ldots, 1, h^*_m, h^*_{m+1}, \ldots, h^*_{\frac{n-1}{2}}, \ldots , h^*_{n-m-1}, h^*_{n-m}, 1, \ldots, 1)
\]
where $m=\frac{1}{2}(n - p_1^{a_1} \cdots p_k^{a_k-1})$ and $h_m > 1$. 
Further, if $\Z_n^*$ denotes the group of units of $\Z_n$, we have that $h^*_{(n-1)/2}\geq n\cdot |\Z_n^*|+1$.
In particular, if $n$ is prime, we have
\[
h^*(T_{C_n}) = (1, \ldots, 1, n^2 - n + 1 , 1, \ldots, 1)  
\]
\end{theorem}

\begin{proof}
Keeping in mind that $n=\kappa$ for $C_n$, denote the height of the lattice point
\[
\frac{(\alpha \1 + \beta (0, 1, \ldots, n-1)) \bmod n}{n}\cdot [L_B \mid \1]
\]
in the fundamental parallelepiped by
\[
h(\alpha,\beta):=\frac{1}{n} \sum_{j=0}^{n-1} ((\alpha + j \beta) \bmod n )\, .
\]
Each $\alpha \in \Z / n \Z$ paired with $\beta = 0$ produces a lattice point at a unique height in $\Pi_{T_{C_n}}$, and thus each $h_i^*\geq 1$.
Let $\Z_n^*$ denote the group of units of $\Z_n$.
If $\beta\in\Z_n^*$, then $\beta (0, 1, \ldots, n-1)\bmod n$ yields a vector that is a permutation of $(0, 1, \ldots, n-1)$, and thus for any $\alpha$ we have the height of the resulting lattice point is $(n-1)/2$, proving that $h^*_{(n-1)/2}\geq n\cdot |\Z_n^*|+1$.
Thus, when $n$ is an odd prime, it follows that
\[
h^*(T_{C_n}) = (1, \ldots, 1, n^2 - n + 1 , 1, \ldots, 1)\, .
\]

Now, suppose that $\gcd(\beta,n)=\prod p_i^{b_i}\neq 1$.
Then the order of $\beta$ in $\Z_n$ is $\prod p_i^{a_i-b_i}$, and (after some reductions in summands modulo $n$)
\[
h(\alpha,\beta)=\frac{1}{n}\cdot \prod p_i^{b_i}\cdot \left(\sum_{j=0}^{\prod p_i^{a_i-b_i}-1}\left((\alpha+j\prod p_i^{b_i})\bmod n\right) \right) \, .
\]
Thus, we see that for a fixed $\beta$, the height is minimized (not uniquely) when $\alpha=0$.
In this case, we have
\begin{align*}
h(0,\beta) & =\frac{1}{n}\cdot \prod p_i^{b_i}\cdot \left(\sum_{j=0}^{\prod p_i^{a_i-b_i}-1}\left(j\prod p_i^{b_i}\bmod n\right) \right) \\
& = \frac{1}{n}\cdot \prod p_i^{b_i}\cdot \prod p_i^{b_i}\cdot\left(\sum_{j=0}^{\prod p_i^{a_i-b_i}-1}j \right) \\
& = \frac{n-\prod p_i^{b_i}}{2} \, .
\end{align*}
This value is minimized when $\prod p_i^{b_i}=p_1^{a_1} \cdots p_k^{a_k-1}$, and this height is attained more than once by setting $\beta=p_1^{a_1} \cdots p_k^{a_k-1}$ and $\alpha=0,1,2,\ldots,p_1^{a_1} \cdots p_k^{a_k-1}-1$.
\end{proof}
 
\begin{corollary} \label{cor:oddcyclenotidp}
$T_{C_n}$ is not IDP for odd $n \ge 3$.
\end{corollary}

\begin{proof} 
Theorem~\ref{primes} yields $h^*_1(T_{C_n}) = 1$ for odd $n \ge 3$.
It is known \cite[Corollary 3.16]{BeckRobinsCCD} that for an integral convex $d$-polytope $\mathcal P$, $h^*_1(\mathcal P) = |\mathcal{P}\cap \Z^n| - (d+1)$.
In this case,
\[
|T_{C_n}\cap \Z^n| = h^*_1(T_{C_n}) + (n-1) + 1 = n+1
\]
is the number of lattice points in $T_{C_n}$.
In particular, the lattice points consist of the $n$ vertices of $T_{C_n}$ and the origin.
Then $\Pi_{T_{C_n}} \cap \{x\mid x_{n}=1\} \cap \Z^n = (0,0, \ldots, 0, 1)$.
If $T_{C_n}$ is IDP, then every lattice point in $\Pi_{T_{C_n}}$ is of the form $(0, \ldots, 0, 1) + \cdots + (0, \ldots, 0, 1),$ which is not true by Proposition \ref{cyclefpp}.
The result follows.
\end{proof}

\section{Complete Graphs}

The simplex $T_{K_n}$ is a generalized permutohedron, where a \emph{permutohedron} $P_n(x_1, \ldots ,x_n)$ for $x_i \in \R$ is the convex hull of the $n!$ points obtained from $(x_1, \ldots , x_n)$ by permutations of the coordinates.
For $K_n$, the Laplacian matrix has diagonal entries equal to $n-1$ and all other entries equal to $-1$.
Then $\conv{L(n)^T} = P_n(n-1, -1, \ldots, -1) \cong P_n(n, 1, \ldots , 1)$. 
Many properties of generalized permutahedra are known \cite{Postnikov}.
While some of the findings in this section follow from these general results, for the sake of completeness we will prove all results in this section from first principles.

\subsection{Reflexivity, Triangulations, and $h^*$-Unimodality}

\begin{theorem}\label{complete} 
The simplices $T_{K_n}$ are reflexive for $n \ge 1$. 
\end{theorem}

\begin{proof} 
Observe $L_B$ is an $n \times (n-1)$ integer matrix of the form

\[
L_B= \begin{bmatrix}
(n-1) & (n-2) & (n-3) & \cdots & \cdots & 1 \\
-1 & (n-2) & (n-3) & \cdots & \cdots & 1 \\
-1 & -2 & (n-3) & \cdots & \cdots & \vdots\\
-1 & -2 & -3 & (n-4) & \cdots & \vdots \\
\vdots & \vdots & \vdots & -4 & \ddots & \vdots \\
\vdots & \vdots & \vdots & \vdots &  & 1 \\
-1 & -2 & -3 & \cdots & \cdots & -(n-1)
\end{bmatrix} 
\]
To prove $T_{K_n}$ is reflexive, we show $T_{K_n} = \{ x \in \R^{n-1} \mid A x \le \1 \}$ for some $A \in \Z^{n \times (n-1)}$. 
We claim that $A$ has the following form:
\[ 
A = 
\left[
\begin{array}{rrrrr}
-1 & 0 & 0 & \cdots & 0 \\
1 & -1 & 0 &  & \vdots \\
0 & 1 & -1 & \ddots & \vdots \\
\vdots & \ddots & \ddots & \ddots &  0 \\
\vdots &  & \ddots & 1 & -1 \\
0 & \cdots & \cdots & 0 & 1
\end{array} 
\right]
\in \{0, \pm 1\}^{n \times (n-1)}. 
\]
Let $r_i$ be the $i$\ss{th} row of $L_B$. 
Observe that $A(i \mid \emptyset) r_i = \1$ for each $1 \le i \le n$. 
Then $\{r_i\}_{i=1}^n$ is a set of intersection points of defining hyperplanes of $T_{K_n}$ taken $(n-1)$ at a time. 
Notice $\rk{A} = n-1$, and further, each matrix $A(i \mid \emptyset)$ has full rank.
This implies $\{r_i \}_{i=1}^n$ is the set of unique intersection points. 
Thus $\{x \mid Ax \le \1 \}=\conv{r_1, r_2, \cdots, r_n} = T_{K_n}$ shows that $T_{K_n}$ is reflexive.
\end{proof}


\begin{prop}\label{triangulation}
The simplex $T_{K_n}$ has a regular unimodular triangulation. \end{prop}

\begin{proof} 
Since the matrix of the facet normals is a signed vertex-edge incidence matrix for a path, it is totally unimodular.
Thus, it follows from \cite[Theorem 2.4]{regulartriangulations} that $T_{K_n}$ has a regular unimodular triangulation. 
\end{proof}

\begin{corollary} \label{cor:completeidp}
The simplex $T_{K_n}$ is IDP. 
\end{corollary}

\begin{proof} 
If $T_{K_n}$ admits a unimodular triangulation, it follows that $T_{K_n}$ is IDP because $\text{cone}(T_{K_n})$ is a union of unimodular cones with lattice-point generators of degree $1$.  
\end{proof}

Theorem~\ref{complete} implies that $h^*(T_{K_n})$ is symmetric.
The following theorem implies that if $\mathcal P$ is reflexive and admits a regular unimodular triangulation, then $h_{\mathcal P}^*$ is unimodal.

\begin{theorem}[Athanasiadis \cite{athanasiadisstable}]\label{unimodal2}
Let $\mathcal P$ be a $d$-dimensional lattice polytope with $h^*_{\mathcal P} = (h_0^*, h_1^*, \ldots , h_d^*)$. 
If $\mathcal P$ admits a regular unimodular triangulation, then $h_i^* \ge h_{d-i+1}^*$ for $1 \le i \le \lfloor (d+1)/2 \rfloor$,
\[
h_{\lfloor (d+1)/2 \rfloor}^* \ge \cdots \ge h_{d-1}^* \ge h_d^*
\]
and 
\[
h_i^* \le \binom{h_1^* + i -1}{i}
\]
for $0 \le i \le d$.
\end{theorem}

\begin{corollary}\label{cor:completeunimodal}
For each $n \ge 2$, $h^*(T_{K_n})$ is unimodal.
\end{corollary}

\subsection{$h^*(T_{K_n})$ and Weak Compositions}

The following is a classification of all lattice points in $\cone{T_{K_n}}$.

\begin{theorem}\label{K_n cone} 
The lattice points at height $h$ in $\cone{T_{K_n}}$ are in bijection with weak compositions of $h\cdot n$ of length $n$, where the height of the lattice point in the cone is given by the last coordinate of the lattice point. 
\end{theorem}

\begin{proof}
Recall the $t$\ss{th} dilate of the polytope $T_{K_n} \subset \R^n$ is given by 
\[  \cone{T_{K_n}} \cap \{z\mid z_n = t\} = \{ \mathbf{\lambda} \cdot [L_B\mid \1] \mid \lambda \in \R_{\ge 0}^n, \sum_{i=1}^n \lambda_i = t \},\] 
since the last coordinate of the lattice point is given by $\sum_{i=1}^n \lambda_i$. 
Notice each lattice point in $\cone{T_{K_n}}$ corresponds uniquely to a lattice point in $tT_{K_n}$ where $t$ is the last coordinate of the point.  
Then the lattice points of $tT_{K_n}$ are all $x = \lambda \cdot [L_B \mid \1] \in \Z^n$ where $0 \le \lambda_i = \dfrac{b_i}{\kappa n}$ for $b_i \in \Z_{\ge 0}$ and $\sum_{i=1}^n \lambda_i = t.$ 
Define the map
\begin{equation*}
\begin{split}
\Phi: \{ \text{length $n$ weak compositions of $tn$} \} &\to \{ \text{lattice points of $tT_{K_n}$} \} \\
c &\mapsto \dfrac{1}{n} c \cdot [L_B \mid \1] \\
\end{split}
\end{equation*}
To show $\Phi(c)$ is a lattice point, consider
\[\Phi(c) = \dfrac{1}{n}[c_1, c_2, \cdots, c_n] \cdot \begin{bmatrix}
(n-1) & (n-2) & (n-3) &  \cdots & 1 & 1 \\
-1 & (n-2) & (n-3) &  \cdots & 1 & 1\\
-1 & -2 & (n-3) &  \cdots & 1 & 1\\
-1 & -2 & -3 &  \cdots & 1 & 1\\
\vdots & \vdots &  \vdots & \ddots & \vdots & \vdots\\
-1 & -2 & -3 &  \cdots & -(n-1) & 1
\end{bmatrix} = \begin{bmatrix}
x_1 \\
x_2 \\
x_3 \\
\vdots \\
\vdots \\
x_n
\end{bmatrix} \, .
 \]
Since $c$ is a weak composition of $tn$, $0 \le \dfrac{c_i}{n} \le t$ for all $i$ and $\frac{1}{n} \sum_{i=1}^n c_i = t$.  
Multiplying the above expression yields $x_i = \left(\sum_{j=1}^i c_j \right) -it$ for all $1 \le i \le n-1$ and $x_n = t$.  
This implies $x \in \Z^n$, which shows $x$ is a lattice point in $tT_{K_n}$. 

To show $\Phi$ is a bijection, we consider the inverse \begin{equation*}
\begin{split}
\Phi^{-1}: \{ \text{lattice points of $tT_{K_n}$} \} &\to \{ \text{length $n$ weak compositions of $tn$} \} \\ 
x &\mapsto n x \cdot [L_B \mid \1]^{-1} \\
\end{split}
\end{equation*}

It can be shown that 
\[
[L_B\mid \1]^{-1} = \dfrac{1}{n} \begin{bmatrix}
1 & -1 & 0 & \cdots & \cdots & 0 \\
0 & 1 & -1 & \ddots &  & \vdots \\
\vdots & 0 & 1 & -1 & \ddots & \vdots \\
\vdots &  & \ddots & \ddots & \ddots & 0 \\
0 & \cdots & \cdots & 0 & 1 & -1 \\
1 & \cdots & \cdots & \cdots & 1 & 1
\end{bmatrix} \, .
\]
Thus
\begin{equation*}
\begin{split}
c &= n x \cdot [L_B \mid \1]^{-1} \\
&= (x_1 + x_n, -x_1 + x_2 + x_n, -x_2 + x_3 + x_n, \ldots , -x_{n-2} + x_{n-1} + x_n, -x_{n-1} + x_n) \\
&= (x_1 + t , -x_1 + x_2 + t, -x_2 + x_3 + t,  \ldots, -x_{n-2} + x_{n-1} + t, -x_{n-1} + t ) \, .
\end{split}
\end{equation*}
It remains to show that $c$ is a weak composition of $tn$. 
First note that $\sum_{i =1}^n c_i = \sum_{i =1}^n t = tn$. 
Next we show each $c_i \ge 0$.  
This is equivalent to $x_1 \ge -t, -x_{n-1} \ge -t,$ and $-x_i + x_{i+1} \ge -t$ for all $2 \le i \le (n-2)$. 
Recall from the hyperplane description of $tT_{K_n}$ that $x$ is lattice point if it satisfies
\[
\begin{bmatrix}
-1 & 0 & \cdots & \cdots & 0 \\
1 & -1 & \ddots &  & \vdots \\
0 & 1 & -1 & \ddots & \vdots \\
\vdots & \ddots & \ddots & \ddots &  0 \\
\vdots &  & \ddots & 1 & -1 \\
0 & \cdots & \cdots & 0 & 1
\end{bmatrix} \cdot \begin{bmatrix}
x_1 \\
x_2 \\
\vdots \\
\vdots \\
x_{n-1} 
\end{bmatrix} \le \begin{bmatrix}
t \\
t \\
t \\
\vdots \\
\vdots \\
t
\end{bmatrix}
\]
These inequalities show $c$ is a weak composition of $tn$ of length $n$.  
Note $\Phi \circ \Phi^{-1} (x) = x$ and $\Phi^{-1} \circ \Phi(c) = c$. 
Thus $\Phi$ is a bijection. 
\end{proof} 

\begin{corollary} 
The Ehrhart polynomial of $T_{K_n}$ is $L_{T_{K_n}}(t) = \binom{tn+n-1}{n-1}.$
\end{corollary}

\begin{proof} 
The number of weak compositions of $tn$ of length $n$ is $\binom{tn+n-1}{n-1}$. 
Then the result follows directly from Theorem~\ref{K_n cone}.
\end{proof}

We next restrict $\Phi$ to obtain a classification of the lattice points in the fundamental parallelepiped, $\Pi_{T_{K_n}}$. 

\begin{corollary}\label{bijection} 
The lattice points of $\Pi_{T_{K_n}}$ are in bijection with weak compositions of $hn$ of length $n$ with each part of size strictly less than $n$. 
\end{corollary}

\begin{proof} 
Every $x \in \Pi_{T_{K_n}} \cap \Z^n $ is of the form $x = \dfrac{1}{\kappa n} b \cdot [L_B \mid \1]$ such that $0 \le \dfrac{b_i}{\kappa n} < 1$ for each $i \in [n]$, i.e., $0 \le \dfrac{b_i}{\kappa} < n$. 
Each coordinate of the lattice point has the form $x_i = \left( \sum_{j=1}^i \dfrac{b_j}{\kappa} \right) - ih$, which is an integer.
It follows by induction on $j$ that $\kappa$ divides $b_j$ for each $1 \le j \le n$.
Then it follows from $\dfrac{1}{\kappa} \sum_{i=1}^n b_i= hn$ that $\left(\dfrac{1}{\kappa} b\right)$ is a weak composition of $hn$ of length $n$ with parts no greater than $n-1$. 

With each $c \in \{\text{length $n$ weak compositions of $tn$ with parts of size less than $n$} \}$, associate $\kappa c = b$. 
This $b$ will generate a lattice point in the fundamental parallelepiped.
The result follows. 
\end{proof}

\begin{prop} \label{prop:completeh*}
For each $n \ge 2$, the $h^*$-vector of $T_{K_n}$ is given by
\[
h^*(T_{K_n}) = (1, m_1, \ldots, m_n)
\] 
where $m_i$ is the number of weak compositions of $in$ of length $n$ with parts of size less than $n$.
\end{prop}
\begin{proof} 
From Lemma~\ref{lem:fpp}, $h_i^*$ enumerates $|\{\Pi_{T_{K_n}} \cap \{ x_n = i\} \cap \Z^n \}|$. By Corollary \ref{bijection}, the result follows.
\end{proof}

\bibliographystyle{plain}
\bibliography{Marie}

\end{document}